\documentclass[12pt,oneside,reqno]{amsart}

\usepackage{accents,graphics,tikz,color,pgfplots,pgfplotstable,subcaption,enumitem}
\usetikzlibrary{calc,plotmarks}

\setlist[enumerate]{leftmargin=1.5pc}
\setlist[enumerate,2]{leftmargin=0.75pc}

\newcommand{\ds}[1]{\displaystyle{#1}}

\newcommand{\dive}{\mathop\mathrm{div}}
\newcommand{\hatdiv}{\hat{\mathrm{div}}}

\newcommand{\curl}{\mathop\mathrm{curl}}

\newcommand{\Hdiv}[1]{H(\mathrm{div},{#1})}
\newcommand{\vi}{{v^{(i)}}}
\newcommand{\vj}{{v^{(j)}}}

\newcommand{\kt}{\tilde{k}^{(i)}}
\newcommand{\ktm}{\tilde{k}^{(i-1)}}
\newcommand{\kto}{\tilde{k}^{(0)}}
\newcommand{\CT}{C_\TT}
\newcommand{\tmax}{t_{\mathrm{max}}}

\newcommand{\diam}{\ensuremath{\mathop\mathrm{diam}}}
\newcommand{\uu}{\mathtt{u}}
\newcommand{\ww}{\mathtt{w}}
\newcommand{\mM}{\mathtt{M}}
\newcommand{\mU}{\mathtt{U}}

\newcommand{\mR}{\mathtt{R}}
\newcommand{\mH}{\mathtt{H}}
\newcommand{\mS}{\mathtt{S}}

\newcommand{\mG}{\mathtt{G}}
\newcommand{\ip}[1]{\langle {#1} \rangle}
\newcommand{\At}{A^{(t)}}
\newcommand{\Aj}{A^{(j)}}

\newcommand{\pp}{{\scriptstyle{P}}}
\newcommand{\tT}{{\scriptstyle{T}}}
\newcommand{\TT}{\mathcal{T}}
\newcommand{\EE}{\mathcal{E}}
\newcommand{\FF}{\mathcal{F}}
\newcommand{\hK}{\hat{K}}

\newcommand{\hD}{\hat{D}}

\newcommand{\vPhi}{\varPhi}

\newcommand{\veps}{\varepsilon}
\newcommand{\vphi}{\varphi}
\newcommand{\divxh}{\hat{\mathrm{div}}_x}
\newcommand{\divx}{{\mathrm{div}}_x}
\newcommand{\gradx}{{\mathrm{grad}}_x}
\newcommand{\divh}{\mathop{\hat{\mathrm{div}}}}
\newcommand{\gradh}{\mathop{\hat{\mathrm{grad}}}}
\newcommand{\RRR}{\mathbb{R}}
\newcommand{\om}{\varOmega}

\numberwithin{equation}{section}
\numberwithin{table}{section}

\theoremstyle{plain}
\newtheorem{theorem}{Theorem}[section]

\theoremstyle{remark}
\newtheorem{remark}{Remark}[section]
\newtheorem{algorithm}{Algorithm}[section]
\newtheorem{example}{Example}[section]

\title[MTP schemes]{Mapped Tent Pitching Schemes for \\ Hyperbolic Systems}

\author{J.~Gopalakrishnan}
\address{Portland State University, PO Box 751, Portland OR 97207, USA}
\email{gjay@pdx.edu}

\author{J.~Sch{\"{o}}berl} 
\address{Wiedner Hauptstra\ss e 8-10, Technische Universit\"at Wien,
  1040 Wien, Austria} 
\email{joachim.schoeberl@tuwien.ac.at}

\author{C.~Wintersteiger}
\address{Wiedner Hauptstra\ss e 8-10, Technische Universit\"at Wien,
  1040 Wien, Austria} 
\email{christoph.wintersteiger@tuwien.ac.at}

\thanks{This work was supported in part by the NSF under DMS-1318916.}

\graphicspath{{mtp_thesis/plots/}{plots/}}

\begin{document}

  \begin{abstract}
    A spacetime domain can be progressively meshed by tent shaped
    objects.  Numerical methods for solving hyperbolic systems using
    such tent meshes to advance in time have been proposed
    previously. Such schemes have the ability 
    to advance in time by different amounts at different
    spatial locations. This paper explores a technique by which 
    standard discretizations, including explicit time stepping, can be
    used within tent-shaped spacetime domains. 
    The technique transforms the equations within a spacetime tent to
    a domain where space and time are separable. After detailing
    techniques based on this mapping, several examples including the
    acoustic wave equation and the Euler system are considered.
  \end{abstract}


\maketitle

\section{Introduction}   \label{sec:intro}

We introduce a new class of methods called Mapped Tent Pitching (MTP)
schemes for numerically solving hyperbolic problems. These schemes can
be thought of as fully explicit or locally implicit schemes on
unstructured spacetime meshes obtained by a process known in the
literature as tent pitching. This process creates an advancing front
in spacetime made by canopies of tent-shaped regions. 
Spacetime tents
are 
erected (with time as the last or vertical dimension in 
spacetime -- see Figure~\ref{fig:tents_levels}) 
so that causality constraints of the hyperbolic problem are 
never violated 
and the hyperbolic problem is solved progressively in layers of tents. 
Such meshing processes were named tent pitching
in~\cite{ErickGuoySulli05,UngorSheff02}. In this paper, we will
refer to tent pitching as a discretization scheme together
with all the attendant meshing techniques. In fact, the main focus of this
paper is not on meshing, but rather on novel discretization techniques that
exploit tent pitched meshes.

\begin{figure}
  \centering
  \begin{subfigure}[t]{0.4\textwidth}
    \includegraphics[width=\textwidth]{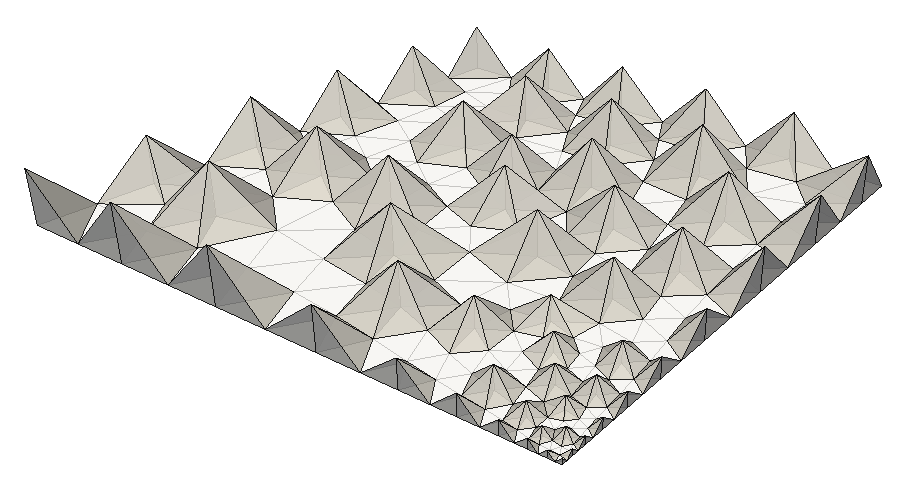}    
    \caption{Initial tents forming Layer~1}
    \label{fig:tents1}
  \end{subfigure}
  \begin{subfigure}[t]{0.4\textwidth}
    \includegraphics[width=\textwidth]{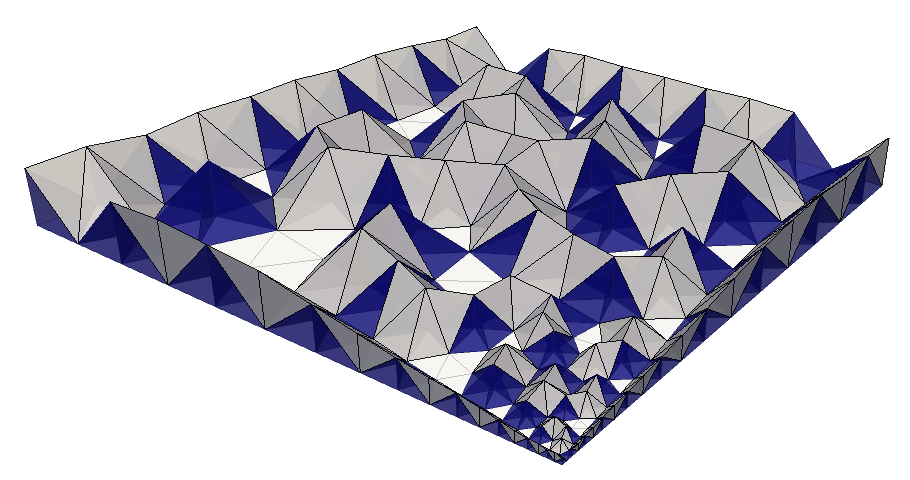}    
    \caption{Layer~2 tents in gray (and layer~1 tents in blue)}
    \label{fig:tents2}
  \end{subfigure}
  \begin{subfigure}[t]{0.4\textwidth}
    \includegraphics[width=\textwidth]{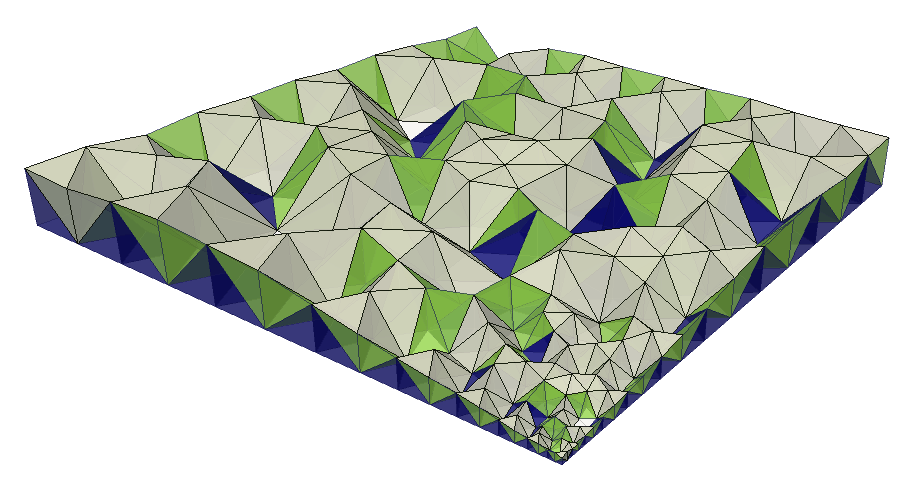}    
    \caption{Layer~3 tents in gray (and tents of previous layers in other colors)}
    \label{fig:tents3}
  \end{subfigure}
  \begin{subfigure}[t]{0.4\textwidth}
    \includegraphics[width=\textwidth]{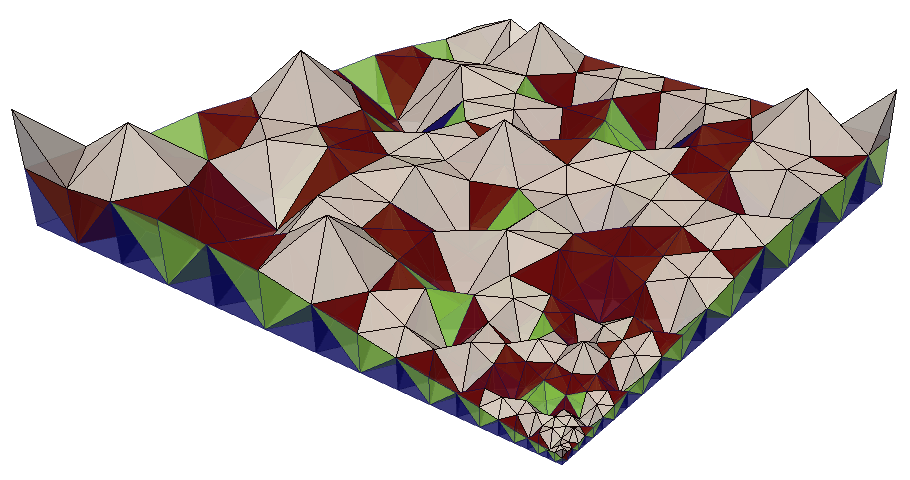}    
    \caption{Layer~4 tents (in gray)}
    \label{fig:tents4}
  \end{subfigure}
  \caption{Parallel tents within different layers}
  \label{fig:tents_levels}
\end{figure}

Today, the dominant discretization technique that use tent pitched
meshes is the spacetime discontinuous Galerkin (SDG) method.  Its
origins can be traced back to~\cite{LowriRoeLeer95,Richt94a}. It has
seen active development over the years in engineering
applications~\cite{MilleHaber08,PalanHaberJerra04,YinAcharSobh00} and
has also motivated several works in numerical
analysis~\cite{FalkRicht99,GopalMonkSepul15,MonkRicht05}. The SDG
schemes use piecewise polynomials in the spacetime elements
(with no continuity constraints
across mesh element interfaces) and a DG (discontinuous Galerkin)
style spacetime discretization. Different
prescriptions of DG fluxes result in different methods. Advanced
techniques, including adaptive spacetime mesh refinement maintaining
causality~\cite{Mont11}, and exact 
conservation~\cite{AbediPetraHaber06}, have been realized for SDG
methods.

The above-mentioned research into SDG methods has abundantly clarified
the many advantages that tent pitched meshes offer.  Perhaps the
primary advantage they offer is a rational way to build high order
methods (in space and time) that incorporates spatial adaptivity and
locally varying time step size, even on complex structures.  Without
tent meshes, many standard methods resort to ad hoc techniques
(interpolation, extrapolation, projection, etc.) for locally adaptive
time stepping~\cite{GandeHalpe13} within inexpensive explicit
strategies. If one is willing to pay the expense of solving global
systems on spacetime slabs~\cite{Neumu13,WangPerss15}, then time and
space adaptivity are easy.  In between these options, there are
interesting alternative methods, without using tents, able to perform
explicit local time stepping while maintaining high order
accuracy~\cite{DiazGrote09,GroteMehliMitko15} by dividing the spatial
mesh into fine and coarse regions.  The concepts we present using
tents provide a different avenue for locally advancing in time.  When
using tents, the height of the tent pole must be restricted to ensure
the solvability of the hyperbolic problem in the tent.  This is
referred to as the {\em causality constraint} and it restricts the
maximal time advance possible at a spatial point. Even if akin to
the Courant-Friedrichs-Levy (CFL) constraint, the causality constraint
does not arise from a discretization and is different from the CFL
constraint.

The main novelty in MTP schemes is a mapping of tents to cylindrical
domains where space and time can be separated, so that standard
spatial discretizations combined with time stepping can be used for
solving on each tent.  MTP schemes proceed as follows: (i) Construct a
spacetime mesh using a tent pitching (meshing) algorithm. (ii) Map the
hyperbolic equation on each tent to a spacetime cylinder.
(iii)~Spatially discretize on the cylinder using any appropriate
existing method. (iv)~Apply an explicit or implicit (high order) time
stepping within the cylinder. (v)~Map the computed solution on the
cylinder back to the tent. Proceeding this way, tent by tent, we
obtain the entire solution.

As a result of this mapping strategy, we are able to construct fully
explicit tent pitching schemes for the first time. We call these {\em
  explicit MTP schemes}.  (Note that the possibility to perform
explicit time stepping within a
tent did not exist with SDG methods.)  Explicit MTP schemes
map each tent to a cylinder, where
space and time can be separated, use a spatial discretization
and thereafter apply an explicit time
stepping to compute the tent solution.  Using explicit MTP schemes, we
are able to bring the well-known cache-friendliness and data locality
properties of explicit methods into the world of local time stepping
through unstructured spacetime tent meshes.  In a later section, we
will show the utility of explicit MTP schemes by applying it to a
complex Mach~3 wind tunnel problem using an existing DG discretization
in space and an explicit time stepping. Note that there is no need to
develop a new spacetime formulation on tents for the Euler system in
order to apply the MTP scheme.

The new mapping strategy also permits the creation of another class
of novel methods which we call {\em locally implicit MTP
  schemes}. Here, after the mapping each tent to a cylinder,
we use an implicit time stepping
algorithm. This requires us to solve a small spatial system (local to
the tent) in order to advance the hyperbolic solution on each tent.
This approach also retains the advantage of being able to use standard
existing spatial discretizations and well-known high order implicit
Runge-Kutta time stepping.  While the explicit MTP schemes are
constrained by both the causality constraint and a CFL constraint
imposed by the choice of the spatial discretization, in locally
implicit MTP schemes there is no CFL constraint.  The causality
constraint applies, and depends on the {\em local} tent geometry and
local wavespeed, but is independent of degree ($p$) of the spatial
discretization. This provides one point of contrast against
traditional methods, whose global timestep
($h_{\mathrm{min}}/p_{\mathrm{max}}^2$) depends on the smallest
element size ($h_{\mathrm{min}}$) and the largest degree
($p_{\mathrm{max}}$) over the entire mesh.

In the remainder of the paper, we will be concerned with hyperbolic
problems that fit into a generic definition described next. Let $L$
and $N$ be integers not less than $1$.  All the problems considered
can be written as a system of $L$ equations on a spacetime
cylindrical domain $\om = \om_0 \times (0,\tmax)$, where the spatial
domain $\om_0$ is contained in $\RRR^N$.  Given sufficiently regular
functions
$
f: \om \times \RRR^L \to \RRR^{L \times N},$
$g: \om \times \RRR^L \to \RRR^L,$
 and $
b : \om \times \RRR^L \to \RRR^L,$
the problem is to find a function $u : \om \to \RRR^L$
satisfying
\begin{equation}
  \label{eq:1}
  {\partial}_t g(x,t,u) + \divx f(x,t,u) + b(x,t,u) = 0
\end{equation}
where ${\partial}_t = {\partial} / {\partial} t$ denotes the time derivative and 
$\divx(\cdot)$ denotes the spatial divergence operator applied row wise
to matrix-valued functions. To be clear, 
the system~\eqref{eq:1}  can be rewritten, using subscripts to denote
components (e.g., $b_l$ denotes the $l$th component of $b$, $f_{li}$
denotes the $(l,i)$th component of $f$, etc.), as
\begin{equation}
  \label{eq:2}
  {\partial}_t g_l(x,t,u(x,t)) 
  + 
  \sum_{i=1}^{N} {\partial}_i (f_{li}(x,t, u(x,t)) + b_l(x,t,u(x,t)) = 0,
\end{equation}
for $  l=1,\ldots, L.$
Here and throughout, ${\partial}_i = {\partial} / {\partial} x_i$ denotes differentiation
along the $i$th direction in $\RRR^{N}$. In examples, we will
supplement~\eqref{eq:2} by initial conditions on $\om_0$ and boundary
conditions on ${\partial}\om_0 \times (0,\tmax).$

We assume that the system~\eqref{eq:1} is {\em hyperbolic in the
  $t$-direction}, as defined in~\cite{Dafer10}. Note that in
particular, this requires that for any fixed $x,t,$ and $u$, the $L
\times L$ derivative matrix $D_ug$ (whose $(l,m)$th entry is ${\partial}
g_m/{\partial} u_l$) is invertible, i.e., 
\begin{equation}
  \label{eq:hyperdet}
  \det[ D_u g] \ne 0.
\end{equation}
Hyperbolicity also provides, for each direction vector and each point
$x,t,u$, a series of real eigenvalues called characteristic
speeds. Let $c(x,t,u)$ denote the maximum of these speeds for all
directions.  For simplicity, we assume that $c(x,t,u)$ is given (even
though it can often be computationally estimated), so that the meshing
process in the next section can use it as input.

Geometrical definitions and meshing algorithms are given in
Section~\ref{sec:tents} (Tents).  Transformation of tents and
hyperbolic equations within them is the subject of
Section~\ref{sec:maps} (Maps).  Two distinct approaches to designing
novel MTP methods are presented in Section~\ref{sec:twoapproaches}. In
Section~\ref{sec:wave}, we discuss a locally implicit MTP method for
the acoustic wave equation in detail. In Section~\ref{sec:euler},
after giving general details pertaining to treatment of nonlinear
hyperbolic conservation laws, we focus on an explicit MTP scheme for
Euler equations.

\section{Tents} \label{sec:tents}

The MTP schemes we present in later sections fall into the category
of methods that use tent pitching for unstructured spacetime
meshing. Accordingly, in this section, we first give a general
description of tent meshing, clarifying the mathematical
meaning of words we have already used colloquially such as ``tent,''
``tent pole,'' ``advancing front,'' etc., and then give details of a
specific meshing algorithm that we have chosen to implement.

\subsection{Overview of a tent pitching scheme}  \label{ssec:pitch}

We now describe how a tent pitching scheme advances the numerical
solution in time. We mesh $\om_0$ by a simplicial conforming shape
regular finite element mesh $\TT.$ The mesh is unstructured to
accomodate for any intricate features in the spatial geometry or in
the evolving solution.  Let $P_1(\TT)$ denote the set of continuous
real-valued functions on $\om_0$ which are linear on each element of
$\TT$.  Clearly any function in $P_1(\TT)$ is completely determined by
its values at the vertices $v_l,$ $l=1,\ldots, N_\TT$, of the
mesh~$\TT$.

At the $i$th step of a tent pitching scheme, the numerical solution is
available for all $x \in \om_0$ and all $0< t <\tau_i(x)$. The
function $\tau_i$ is in $P_1(\TT)$.  The graph of $\tau_i$, denoted by
$S_i$, and is called the ``advancing front'' (see
Figure~\ref{fig:tents_levels}.)  We  present a serial version of
the algorithm first. A parallel generalization is straightforward as
mentioned in Remark~\ref{rem:partents}.  A tent pitching scheme
updates $\tau_i$ within the following general algorithmic outline:

\begin{algorithm} \hfill
\begin{enumerate} 
\item    Initially, set $\tau_0 \equiv 0$. Then $S_0=\om_0$. The solution on
   $S_0$ is determined by the initial data on $\om_0$. 

\item For $i=1,2,\ldots,$ do:
  \begin{enumerate}

  \item Find a mesh vertex $\vi$ where good relative progress in
    time can be made and calculate the height (in time) $k_i$ by which
    we can move the advancing front at $v_i$. One strategy to do this
    is detailed below in Algorithm~\ref{alg:pitchloc}.

  \item Given the solution on the current advancing front $S_{i-1}$, 
    pitch a ``spacetime tent'' $K_i$ by erecting a ``tent pole'' of
    height $k_i$ at the point $(\vi, \tau_{i-1}(\vi))$ on~$S_{i-1}$.
    Let $\eta_i \in P_1(\TT)$ be the unique function that equals one
    at $\vi$ and is zero at all other mesh vertices.  Set
    \begin{equation}
      \label{eq:tau}
      \tau_i = \tau_{i-1} + k_i \eta_i      
    \end{equation}
    Define the ``vertex patch'' $\om_v$ of a mesh vertex $v$
    as the (spatial) open
    set in $\RRR^N$ that is the interior of the union of all simplices
    in $\TT$ connected to~$v$.  Then the tent $K_i$ can be expressed as
    \begin{equation}
      \label{eq:Ki}
      K_i = \{ (x,t): \; x \in \om_{\vi},\; \tau_{i-1}(x) < t < \tau_i(x)\}.
    \end{equation}

  \item \label{item:3} Numerically solve~\eqref{eq:1} on $K_i$ (e.g., by the
    methods proposed in the later sections of this paper).  Initial
    data is obtained from the given solution on $S_{i-1}$. If $\vi \in
    {\partial}\om_0$, then the boundary conditions required to
    solve~\eqref{eq:1} on $K_i$ are obtained from the given boundary
    conditions on the global boundary ${\partial}\om_0 \times (0,\tmax)$.

  \item If $\tau(v) \ge \tmax$ for all mesh vertices $v$, then exit.
  \end{enumerate}
\end{enumerate}
\end{algorithm}

The height of the tent pole $k_i$ at each step should be determined
using the causality constraint so that~\eqref{eq:1} is solvable on
$K_i$. The choice of the vertex $\vi$ should be made considering the
height of the neighboring vertices. Other authors have studied these
issues~\cite{ErickGuoySulli05,UngorSheff02} and given appropriate
advancing front meshing strategies. Next, we describe a simple
strategy which we have chosen to implement. It applies verbatim in
both two and three space dimensions.

\subsection{Algorithm to mesh by tents}

To motivate our meshing strategy, first let $\bar c(x)$ denote a given
(or computed) approximation to the maximal characteristic speed at a
point $(x,\tau_{i-1}(x))$ on the advancing front $S_{i-1}$, e.g.,
$\bar c(x) = c(x,\tau_{i-1}(x), u(x,\tau_{i-1}(x))$, where $u$ is the
computed numerical solution. We want to ensure that, for all
$x \in \om_0$, we have
\begin{equation}
  \label{eq:cfl1}
| \gradx \tau_{i}(x) | \;\le\; \frac{1}{\bar c (x)}
\end{equation}
at every step $i$. Here $| \cdot|$ denotes the Euclidean norm. 
This is our CFL condition.

For simplicity, we now assume that $c$ is independent of time and
impose the following CFL condition which is more stringent
than~\eqref{eq:cfl1}:
\begin{equation}
  \label{eq:cfl2}
  | \gradx \tau_{i} |\big|_T \;\le\; \frac{1}{c_T} ,\qquad \text{ for all } 
  T \in \TT,
\end{equation}
where $c_T = \max_{x\in T} \bar c(x).$ In practice, it is easier to
work with the following sufficient condition
\begin{equation}
  \label{eq:13}
  \frac{\tau_i(e_1) - \tau_i(e_2)}{|e|} \le \frac{\CT}{c_e},
  \qquad \text{ for all mesh edges } e,
\end{equation}
where $e_1$ and $e_2$ are the mesh vertices that are endpoints of the
edge $e$ of length $|e|$, $c_e$ is the maximum of $c_T$ over all
elements $T$ which have $e$ as an edge, and $\CT$ is a constant that
depends on the shape regularity of the mesh $\TT$. It is easy to see
that~\eqref{eq:13} implies~\eqref{eq:cfl2}.

To obtain an advancing front satisfying~\eqref{eq:13} at all
stages~$i$, we maintain a list of potential time advance $\kt_l$ that
can be made at any vertex $v_l$. Let $\EE_l$ denote the set of all
mesh edges connected to the vertex $v_l$ and suppose edge endpoints
are enumerated so that $e_1 = v_l$ for all $e \in \EE_l$.  Given
$\tau_i$ satisfying~\eqref{eq:13}, while considering pitching a tent
at $(v_l, \tau_{i}(v_l))$ so that~\eqref{eq:13} continues to hold, we
want to ensure that
\[
\frac{ \left( \tau_{i}(v_l) + \kt_l\right)  - \tau_{i}(e_2 ) }{|e|}
\le \frac{ \CT}{c_{e}}
\quad\text{ and }\quad
\frac{ -\left( \tau_{i}(v_l) + \kt_l\right)  + \tau_{i}(e_2 ) }{|e|}
\le \frac{ \CT}{c_{e}}
\]
hold for all $e \in \EE_l$. The latter inequality is obvious
from~\eqref{eq:13} since we are only interested in $\kt_l\ge 0$.
The former inequality is ensured if we choose
\[
\kt_l \le
      \min_{e \in \EE_l}
      \left( \tau_{i}(e_2) - \tau_{i}(v_l)  + |e|\frac{\CT}{c_e}
      \right), 
\]
as done in the Algorithm~\ref{alg:pitchloc} below. The algorithm also
maintains a list of locations ready for pitching a tent. For this, it
needs the reference heights
$
r_l = \min_{e \in \EE_l} |e|  \CT/ c_e 
$
(the maximal tent pole heights on a flat advancing front) which can be
precomputed. Set $\kto_l = r_l$.  A vertex $v_l$ is considered a
location where ``good'' progress in time can be made if its index $l$
is in the set
\begin{equation}
  \label{eq:Ji}
  J_i = \left\{ l:\;  \kt_l \ge  \gamma r_l \right\}.
\end{equation}
Here $0 < \gamma <1$ is a parameter (usually set to $1/2$). While a
lower value of $\gamma$ identifies many vertices to progress in time
moderately, a higher value of $\gamma$ identifies fewer vertices where
time can be advanced more aggressively.

\begin{algorithm}
  \label{alg:pitchloc}
  Initially, $\tau_0 \equiv 0$, $\kto_l = r_l$ and $J_0 = \{1,2\cdots,
  N_\TT\}$. For $i \ge 1$, given $\tau_{i-1}$, $\{ \ktm_l\}$, and
  $J_{i-1}$, we choose the next tent pitching location ($\vi$) and the
  tent pole height ($k_i$), and update as follows:
  \begin{enumerate}
  \item Pick any $l_*$ in $J_{i-1}$.
  \item Set $ \vi = v_{l_*}$ and $ k_i = \ktm_{l_*}.
    $
  \item Update $\tau_i$ by~\eqref{eq:tau}. 
  \item Update $\kt_l$ for all vertices
    $v_l$ adjacent to $\vi$ by
    \[
    \kt_l = \min \left(
      \tmax - \tau_{i}(v_l),\;
      \min_{e \in \EE_l}
      \left( \tau_{i}(e_2) - \tau_{i}(v_l)  + |e|\frac{\CT}{c_e}
      \right) 
      \right).
    \]
  \item Use~$\{\kt_l\}$ to set $J_i$ using~\eqref{eq:Ji}.
  \end{enumerate}
\end{algorithm}

\begin{remark}[Parallel tent pitching] \label{rem:partents}
  To pitch multiple tents in parallel, at the $i$th step, instead of
  picking $l_*$ arbitrarily as in Algorithm~\ref{alg:pitchloc}, we
  choose $l_* \in J_{i-1}$ with the property that
  $ \om_{v_{l_*}} = \om_{\vi}$ does not intersect $\om_\vj$ for all
  $j < i$. As we step through $i$, we continue to pick such $l_*$
  until we reach an index $i=i_1$ where no such $l_*$ exists. All the
  tents made until this point, say $K_1, K_2, \ldots, K_{i_1}$ form
  the {\em layer} $L_1$. (An example of tents within such layers are
  shown in Figure~\ref{fig:tents_levels} -- in this example one of the
  corners of the domain  has a singularity.)  We then repeat this
  process to find greater indices $i_2< i_3 < \cdots$ and layers
  $L_k = \{ K_{i_{k-1}}, K_{i_{k-1}+1}, \ldots, K_{i_k} \}$ with the
  property that $\om_{\vj}$ does not intersect $\om_{\vi}$ for any
  distinct $i$ and $j$ in the range $i_{k-1} \le i,j \le i_k$.
  Computations on tents within each layer can proceed in parallel.
\end{remark}

\section{Maps}   \label{sec:maps}

In this section we discuss a mapping technique that allows us to
separate space and time discretizations within tents. 
Domains like $\om_0 \times (0,T)$ formed by a tensor product of a
spatial domain with a time interval are referred to as spacetime
cylinders. Such domains are amenable to tensor product
discretizations where the space and time discretizations neatly
separate. However, the tent $K_i$ in~\eqref{eq:Ki} is not of this
form.  Therefore, we now introduce a mapping that transforms $K_i$
one-to-one onto the spacetime cylinder
$\hK_i = \om_{\vi} \times (0, 1).$

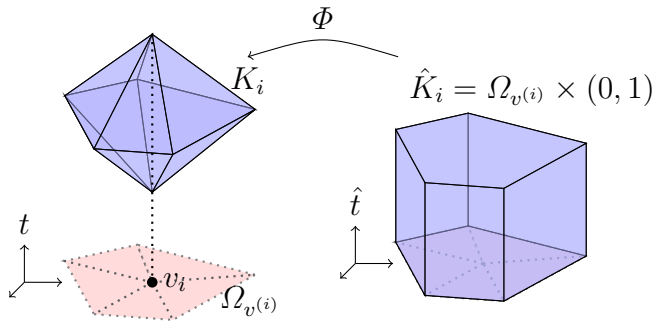
\begin{figure}
  \centering
  \begin{tikzpicture}
    \tikzstyle{blufill} = [fill=blue!30,fill opacity=0.5]
    \tikzstyle{shinyfill} = [fill=red!20, fill opacity=1]
    \tikzstyle{shinybehind}=[draw=black!70, draw opacity=0.7, fill=red!20, fill opacity=0.7]

    \coordinate (a0) at (0,0);
    \coordinate (a1) at (0,1.1);
    \coordinate (a2) at (-1.162,.283);
    \coordinate (a3) at (-.2,.5);
    \coordinate (a4) at (-.775,-.424); 
    \coordinate (a5) at (1.369,.1);
    \coordinate (a6) at (.274,-.5);
    \coordinate (b1) at (0,-1.);

    \coordinate (K) at ($(a5)+(-0.1,0.7)$);
    \node[below] at (K) {$K_i$};


    \filldraw[blufill] (a1)--(a2)--(a3) --cycle;   
    \filldraw[blufill] (a1)--(a3)--(a5) --cycle;

    \filldraw[blufill] (a2)--(a3)--(a5)--(a6)--(a4)--cycle;

    \filldraw[blufill] (b1)--(a2)--(a4) --cycle;
    \filldraw[blufill] (b1)--(a2)--(a3) --cycle;
    \filldraw[blufill] (b1)--(a5)--(a3) --cycle;

    \filldraw[blufill] (a1)--(a2)--(a4)--cycle;
    \filldraw[blufill] (a1)--(a4)--(a6)--cycle;
    \filldraw[blufill] (a1)--(a6)--(a5)--cycle;

    \filldraw[blufill] (b1)--(a6)--(a4)--cycle;
    \filldraw[blufill] (b1)--(a6)--(a5)--cycle;


    
    \draw[dotted,thick] (b1)--(a1);

    \coordinate (b0) at ($(a0)-(0,2.2)$);
    \coordinate (b2) at ($(a2)-(0,2.2)$);
    \coordinate (b3) at ($(a3)-(0,2.2)$);
    \coordinate (b4) at ($(a4)-(0,2.2)$);
    \coordinate (b5) at ($(a5)-(0,2.2)$);
    \coordinate (b6) at ($(a6)-(0,2.2)$);

    \filldraw[shinybehind,dotted,thick] (b0)--(b2)--(b4)--cycle;
    \filldraw[shinybehind,dotted,thick] (b0)--(b2)--(b3)--cycle;
    \filldraw[shinybehind,dotted,thick] (b0)--(b5)--(b3)--cycle;
    \filldraw[shinybehind,dotted,thick] (b0)--(b5)--(b6)--cycle;
    \filldraw[shinybehind,dotted,thick] (b0)--(b4)--(b6)--cycle;

    \draw[dotted,thick] (b1)--(b0);

    \coordinate (o) at ($(b0)+(-1.7,0)$);
    \draw [->] (o)-- ++(0.5,0);
    \draw [->] (o)-- ++(-0.2,-0.2);
    \draw [->] (o)-- ++(0,0.5) node[above] {$t$};

    \fill (b0) circle (2pt) node[above,right] {$v_i$}; 
    \node at ($(o)+(3,-0.25)$) {$\om_{\vi}$};


    \coordinate (a0) at ($(a0)+(4.4,-1.2)$);
    \coordinate (a2) at ($(a2)+(4.4,-1.2)$);
    \coordinate (a3) at ($(a3)+(4.4,-1.2)$);
    \coordinate (a4) at ($(a4)+(4.4,-1.2)$);
    \coordinate (a5) at ($(a5)+(4.4,-1.2)$); 
    \coordinate (a6) at ($(a6)+(4.4,-1.2)$);  
    \coordinate (a1) at ($(a1)+(4.4,-1.2)$); 
    \coordinate (b1) at ($(b1)+(4.4,-1.2)$);

    \coordinate (t0) at ($(a0)+(0,0.75)$);
    \coordinate (t2) at ($(a2)+(0,0.75)$);
    \coordinate (t3) at ($(a3)+(0,0.75)$);
    \coordinate (t4) at ($(a4)+(0,0.75)$);
    \coordinate (t5) at ($(a5)+(0,0.75)$);
    \coordinate (t6) at ($(a6)+(0,0.75)$);

    \coordinate (b0) at ($(a0)-(0,0.75)$);
    \coordinate (b2) at ($(a2)-(0,0.75)$);
    \coordinate (b3) at ($(a3)-(0,0.75)$);
    \coordinate (b4) at ($(a4)-(0,0.75)$);
    \coordinate (b5) at ($(a5)-(0,0.75)$);
    \coordinate (b6) at ($(a6)-(0,0.75)$);





    \filldraw[shinybehind,dotted,thick] (b0)--(b2)--(b4)--cycle;
    \filldraw[shinybehind,dotted,thick] (b0)--(b2)--(b3)--cycle;
    \filldraw[shinybehind,dotted,thick] (b0)--(b5)--(b3)--cycle;
    \filldraw[shinybehind,dotted,thick] (b0)--(b5)--(b6)--cycle;
    \filldraw[shinybehind,dotted,thick] (b0)--(b4)--(b6)--cycle;

    \filldraw[blufill] (b2)--(b4)--(b6)--(b5)--(b3)--cycle;

    \filldraw[blufill] (t2)--(b2)--(b3)--(t3)--cycle;
    \filldraw[blufill] (t3)--(b3)--(b5)--(t5)--cycle;
    \filldraw[blufill] (t5)--(b5)--(b6)--(t6)--cycle;
    \filldraw[blufill] (t6)--(b6)--(b4)--(t4)--cycle;
    \filldraw[blufill] (t2)--(b2)--(b4)--(t4)--cycle;

    \filldraw[blufill] (t2)--(t4)--(t6)--(t5)--(t3)--cycle;

    \coordinate (o) at ($(b0)+(-1.7,0)$);
    \draw [->] (o)-- ++(0.5,0);
    \draw [->] (o)-- ++(-0.2,-0.2);
    \draw [->] (o)-- ++(0,0.5) node[above] {$\hat t$};
    
    \coordinate (Kh) at ($(K)+(2,0)$);
    \node[anchor=north west] at (Kh) {$\hat K_i = \om_{\vi} \times (0,1)$};
    \coordinate (Kc) at ($(K)+(1,0.3)$);
    \draw[->] (Kh) .. controls (Kc) ..(K) node[midway,above] {$\vPhi$};    
  \end{tikzpicture}
  \caption{Tent mapped from a tensor product domain.\label{fig:map}}
\end{figure}

Define the mapping $\vPhi: \hK_i \to  K_i$ (see Figure~\ref{fig:map})  by 
$
\vPhi(\hat x, \hat t) 
= (\hat x, \vphi(\hat x,\hat t)),$ where
$\vphi(\hat x,\hat t) =    (1-\hat t )  \tau_{i-1}(\hat x) + \hat t \tau_i(\hat x), 
$
for all $ (\hat x, \hat t)$ in $\hK_i$. Note that the $(N+1) \times
(N+1)$ Jacobian matrix of derivatives of $\vPhi$ takes the form
\begin{equation}
  \label{eq:DPhi}
\hD \vPhi = 
\begin{bmatrix}
  I  & 0 
  \\
  \hat D \varphi & \delta
\end{bmatrix}
\end{equation}
where 
$
\hat D \varphi = 
[\gradh \vphi]^t
= 
[ 
  \hat{\partial}_1\vphi\;\,
  \hat{\partial}_2\vphi\; \cdots \;
  \hat{\partial}_N\vphi
],$
and $
\delta = \tau_i - \tau_{i-1}.
$
Here and throughout, we use abbreviated notation for derivates 
$
\hat {\partial}_j ={\partial} /{\partial} \hat x_j,$
$\hat {\partial}_t = {\partial}/ {\partial} \hat t = \hat {\partial}_{N+1}
$
that also serves to distinguish differentiation on $\hat K_i$ from differentiation
(${\partial}_i$) on $K_i$. Define 
\begin{subequations}
  \label{eq:hatfgb}
\begin{align}
  \hat f( \hat x, \hat t, w ) & = f ( \vPhi(\hat x, \hat t), w),
  & 
  \hat g  (\hat x, \hat t, w) & = g ( \vPhi(\hat x, \hat t), w ),
  \\
  \hat b ( \hat x, \hat t, w ) & = b ( \vPhi(\hat x, \hat t), w),
  & 
  \hat G(\hat x, \hat t, w)  & 
                = \hat g( \hat x, \hat t, w) 
    - \hat f( \hat x, \hat t, w) \gradh \varphi 
  \\
  \hat u  & = u\circ \vPhi,
  & 
  \hat U( \hat x, \hat t) & = \hat G( \hat x, \hat t, \hat u(\hat x, \hat t)).
\end{align}
\end{subequations}
The last equation, showing that the function
$\hat u : \hat K_i \to \RRR^L$ is mapped to
$\hat U : \hat K_i \to \RRR^L$ by $\hat G$, will often be abbreviated
as simply $ \hat U = \hat G (\hat u)$.

\begin{theorem}
  \label{thm:map}
  The function $u$ satisfies~\eqref{eq:1} in $K_i$ if and only if 
  $\hat u $ and $\hat U$
  satisfy 
  \[
  \hat {\partial}_t \hat U
  + 
  \divh \big( \delta \hat f \,\big) 
  + 
  \delta \, \hat b = 0
    \qquad \text{ in } \hat K_i,
  \]
  which in expanded component form reads as 
  \begin{equation}
    \label{eq:mapped}
      \hat {\partial}_t
      [\hat G( \hat u)]_l+
  \sum_{j=1}^N \hat {\partial}_j  \left( 
    \delta(\hat x) 
    \hat f_{lj} (\hat x, \hat t, \hat u(\hat x, \hat t)) 
  \right)
  + 
  {\partial}elta(\hat x) \hat b_l (\hat x, \hat t, \hat u( \hat x, \hat t)) =0
  \end{equation}
  for all $(\hat x, \hat t) $ in $\hat K_i$ and all $l =1, \ldots, L.$
\end{theorem}
\begin{proof}
  The proof proceeds by calculating the 
  pull back of the system~\eqref{eq:1} from
  $K_i$ to $\hat K_i$ using the map $\vPhi$. 
  Using the given $u$, define $F_l : K_i \to
  \RRR^{N+1}$ and $B: K_i \to \RRR^{L}$ by
\[
F_l (x,t) = 
\begin{bmatrix}
  f_{l1}( x, t, u(x,t)) \\
  \vdots \\
  f_{lN}( x, t, u(x,t)) \\
  g_{l}( x,t, u(x,t))
\end{bmatrix},
\quad 
B(x,t) = 
\begin{bmatrix}
  b_1( x,t,  u(x,t))\\
  \vdots \\
  b_{L}( x,t, u(x,t))
\end{bmatrix}
\]
and define their pullbacks on $\hK_i$ by 
$
  \hat F_l  = {\partial}et[\hD\vPhi]\, [\hD\vPhi]^{-1} (F_l \circ \vPhi)
$ and $
  \hat B  = {\partial}et[\hD\vPhi]\, (B \circ \vPhi).
$
By the well-known properties of the Piola map (or by direct
computation),
\begin{equation}
  \label{eq:4}
  \hatdiv{\hat F} = 
  {\partial}et[\hD\vPhi]\, 
  ({\partial}ive F) \circ \vPhi,
\end{equation}
where the divergence on either side is now taken in spacetime
($\RRR^{N+1}$).  Note that $\det [\hD\vPhi] = \delta $ is never zero
at any point of (the open set) $\hat K_i$.  Writing
equation~\eqref{eq:1} in these new notations, we obtain
$ (\dive F_l)(x,t) + B(x,t) = 0$ for all $(x,t) \in K_i,$ or
equivalently,
\begin{align*}
  (\dive F_l)(\vPhi(\hat x, \hat t)) + B(\vPhi(\hat x, \hat t)) 
  & = 
    0
\end{align*}
for all $(\hat x, \hat t) \in \hat K_i.$
Multiplying through by
$\det[\hD\vPhi]$ and using~\eqref{eq:4}, this becomes
\begin{equation}
  \label{eq:3}
  \hatdiv \hat F_l + \hat B = 0, 
  \qquad \text{ on } \hat K_i.
\end{equation}

To finish the proof, we simplify this equation.
Inverting the block triangular matrix $\hD \vPhi$ displayed
in~\eqref{eq:DPhi} and using it in the definition for $\hat F_l$, we
obtain
\[
\hat F_l = \det[\hD\vPhi] 
\begin{bmatrix}
  I & 0 
  \\
  -\delta^{-1}
  \hat D \varphi 
  & 
  \delta^{-1}
\end{bmatrix} 
F_l \circ \vPhi
= 
\begin{bmatrix}
   \delta  \hat f_l 
  \\
  \hat g_l
  - \gradh \varphi \cdot \hat f_l 
\end{bmatrix}
\]
where $\hat f_l$ is the vector whose $i$th component is
$\hat f_{li}(\hat x, \hat t, \hat u)$ and $\hat g_l$ denotes the $l$th
component of $\hat g(\hat x, \hat t, \hat u)$. Substituting these
into~\eqref{eq:3} and expanding, we obtain~\eqref{eq:mapped}.
\end{proof}

\section{Two approaches to MTP schemes}  \label{sec:twoapproaches}

Theorem~\ref{thm:map} maps the hyperbolic system to the cylinder which
is a tensor product of a spatial domain $\om_{\vi}$ with a time
interval $(0,1)$. This opens up the possibility to construct tensor
product discretizations -- rather than spacetime discretizations --
within each tent.

We denote by $\TT_i$ the spatial mesh of
$\om_{\vi}$ consisting of elements of $\TT$ having $\vi$ as a vertex.
For the spatial discretization, we use a finite element space~$X_i$
based on the mesh $\TT_i$. In order to discretize~\eqref{eq:mapped},
we multiply it with a spatial test function $v$ in $X_i$, integrate
over the vertex patch $\om_{\vi}$, and 
manipulate the terms to get an equation of the form 
\begin{equation}
  \label{eq:14}
\int_{\om_{\vi}}\hat {\partial}_t \hat U( \hat x, \hat t) \cdot v (\hat x) \; d\hat x  
= S_i( \hat t, \hat u, v),   
\end{equation}
for all $\hat t \in (0,1)$ and $v \in X_i$. Details of the spatial
discretization, yet unspecified, are lumped into $S_i$.  
Note that the temporal
derivative occurs only in the first term and can be discretized using
Runge-Kutta or other schemes.  Emphasizing the point that spatial
discretization is thus separated from temporal discretization, 
we continue, leaving time  undiscretized,
to discuss two semidiscrete approaches.

\subsection{First approach}  \label{ssec:first}

Recalling that $\hat U$ depends on $\hat u,$ the first approach 
discretizes $\hat u(\cdot, \hat t)$ in $X_i$.
Let the functions $\psi_{n} : \om_{\vi} \to \RRR^L$, for
${n} = 1,\ldots, P,$ form a basis of $X_i$. We seek an approximation to
$\hat u$ of the form
$
\hat u_h(\hat x, \hat t) = \sum_{{n}=1}^P 
\uu_{n}(\hat t) \,\psi_{n}(\hat x)
$
where $\uu (\hat t)$, 
the vector whose ${n}$th entry is $\uu_{n}(\hat t)$,
is to be found. Substituting this into~\eqref{eq:14} and
using~\eqref{eq:hatfgb}, we obtain
$\int_{\om_{\vi}}\hat {\partial}_t \hat G( 
\hat u_h) \cdot v  \; d\hat x  
= S_i( \hat t, \hat u_h, v),   
$
for all $v \in X_i$ and $\hat t $ in $(0,1)$. To view this as a
finite-dimensional system of ordinary differential equations (ODEs),
define two maps $\mG$ and $\mS$ on $\RRR^P$ by
\[
[\mG(\ww)]_{m} = 
\int_{\om_{\vi}} \hat G
\left( 
\sum_{{n}=1}^P \ww_{n} \psi_{n}(\hat x) \right) \psi_m(\hat x)\, d\hat x, 
\quad
[\mS(\ww)]_m = S_i\bigg(\hat t, 
\sum_{{n}=1}^P \ww_{n} \psi_{n} , \;\psi_m\bigg).
\]
Then, putting $v = \psi_{n}$, 
we obtain the
semidiscrete problem of finding a $\uu: (0,1) \to \RRR^P$, given
initial values $\uu(0)$, satisfying the ODE system
\begin{equation}
  \label{eq:appr-1}
\frac{d}{d \hat t} \mG( \uu(\hat t) ) = 
\mS (\uu(\hat t)), \qquad 0 < \hat t < 1.  
\end{equation}

\subsection{Second approach}  \label{ssec:second-approach}

The  second approach discretizes $\hat U$ rather than $\hat u$,
assuming that~$\hat G^{-1}$ is at hand. 
We substitute $ \hat u = \hat G^{-1}( \hat U)$ into the right
hand side of~\eqref{eq:14} and obtain the following semidiscrete
problem. Find $\hat U_h$ of the form
\begin{equation}
  \label{eq:secondapp}
\hat U_h(\hat x, \hat t) = \sum_{{n}=1}^P \mU_{n}(\hat t) \psi_{n}(\hat x)  
\end{equation}
that satisfies 
$
\int_{\om_{\vi}}\hat {\partial}_t \hat U_h\cdot v  \; d\hat x  
= S_i( \hat t, \hat G^{-1}(\hat U_h), v), 
$
for all $v \in X_i$ and $\hat t $ in $(0,1)$. With 
\[
\mM_{m{n}} = 
\int_{\om_{\vi}} 
\psi_{n}(\hat x) \psi_m(\hat x)\, d\hat x, 
\quad
[\mR(\ww)]_m = S_i \left( \hat t, 
\hat G^{-1} \left( \sum_{{n}=1}^P \ww_{n} \psi_{n}\right) , \psi_m\right).
\] 
we obtain the following ODE system for $\mU$, the vector
of coefficients $\mU_{n}(t)$.
\begin{equation}
  \label{eq:appr-2}
  \frac{d}{d \hat t} \mM  \mU(\hat t) = 
  \mR (\mU(\hat t)), \qquad 0 < \hat t < 1.
\end{equation}
Comparing with~\eqref{eq:appr-1}, instead of 
a possibly nonlinear $\mG$,
we now have a linear 
 action of the mass matrix $\mM$ in
$\RRR^{P\times P}.$ 


\subsection{Examples}

We first illustrate how to treat 
a very general linear hyperbolic system using the
first approach. In the second example we illustrate the 
second approach using   a simple nonlinear
conservation law.

\begin{example}[Linear hyperbolic systems] \label{eg:linear} Suppose
  that $\Aj: \om_0 \to \RRR^{L \times L}$, for $j=1,\ldots, N$, are 
  symmetric matrix-valued functions and 
  $B: \om \to \RRR^{L \times L}$ is bounded.  In addition, suppose
  $\At \equiv A^{(N+1)}$ is a symmetric positive definite
  matrix-valued function from $\om_0$ to $\RRR^{L \times L}$. 
  A large class of
  linear examples can be obtained by setting
\begin{align}  \label{eq:linear}
[f(x,t,u)]_{lj} & = \sum_{m=1}^L \Aj_{lm}(x) u_m,
& 
[g(x,t,u)]_{l} & =  \sum_{m=1}^L \At_{lm}(x) u_m.
\end{align}    
and $b(x,t,u)  = B(x,t) u.$
Then~\eqref{eq:1} can be written as 
\begin{equation}
  \label{eq:5}
  {\partial}_t (\At u) + \sum_{j=1}^N {\partial}_j( \Aj u) + B u = 0.
\end{equation}

A simple equation that fits into this example is the scalar {\em
  transport equation}. The  transport of a scalar density $u$
along a given divergence-free vector field $\beta : \om_0 \to \RRR^N$
is described by
$  {\partial}_t u + \dive ( \beta u) =0.
$
This fits in the setting of~\eqref{eq:5} with $L=1$, $B=0$, $ \At =
[1],$ and $ A^{(j)}(x) = [\beta_j(x)],$ for $j=1,2,\ldots,N$.  
A more complex system that also fits into this example is {\em
  electromagnetic wave propagation}.  Given positive functions
$\veps$, $\mu$ and $\sigma$ on $\om_0$, the Maxwell system for
electric field $E$ and magnetic field $H$ consists of
$    \veps {\partial}_t E - \curl H  + \sigma E  = 0
$ and $
    \mu {\partial}_t H + \curl E  =0.$
This system also fits into~\eqref{eq:5} with 
$  N=3,$ $ L=6,$ and 
$ u =
[\begin{smallmatrix}
    E \\ H 
  \end{smallmatrix}]
$
and 
\begin{align*}
  \Aj =
  \begin{bmatrix}
    0 & [\epsilon^j] \\
    [\epsilon^j]^t & 0 
  \end{bmatrix},\quad
                     \At =
                     \begin{bmatrix}
                       \veps I &  0 \\
                       0       & \mu I 
                     \end{bmatrix}, \quad
                                 B =
                                 \begin{bmatrix}
                                   \sigma & 0 \\
                                   0 & 0
                                 \end{bmatrix}
\end{align*}
where $\epsilon^j$ is the matrix whose $(l,m)$th entry is the
alternator $\epsilon_{jlm}.$

To solve~\eqref{eq:5} 
for $u$ on a spacetime tent $K_i$,
  we first map~\eqref{eq:5} to the 
  spacetime cylinder $\hat K_i$ using Theorem~\ref{thm:map}. We
  find that the map $\hat u \to \hat U$ is now given by
  $
  \hat U = \hat G(\hat u) = 
  H(\hat x, \hat t) \hat u
  $
  where $H : \hat K_i \to \RRR^{L \times L}$ is the matrix function
  \begin{equation}
    \label{eq:H}
    H = 
    \At - 
    \sum_{j=1}^N\bigg[ (1-\hat t) \hat \partial_j\tau_{i-1} + \hat t\hat{\partial}_j \tau_i
    \bigg] \Aj.
  \end{equation}
  Following the first approach, we  discretize the term
  $\hat {\partial}_t ( H \hat u )$ in that form.
  The semidiscretization~\eqref{eq:appr-1} now takes the form
  \begin{equation}
    \label{eq:H-approach-1}
    \frac{d}{d \hat t} \left( \mH(\hat t)  \uu(\hat t)  \right)= 
    \mS (\uu(\hat t)), \qquad 0 < \hat t < 1,
  \end{equation}
  where $\mH$ is the matrix whose entries are
  $\displaystyle{
    \mH_{m{n}}(\hat t) = 
    \int_{\om_{\vi}} 
    H(\hat x, \hat t)\, \psi_{n} (\hat x) \cdot \psi_m(\hat x)
    \, d\hat x.
  }$
\end{example}

\begin{example}[2D inviscid scalar Burgers equation] \label{eg:burger}
  A simple two-dimensional analogue of the well-known one-dimensional
  inviscid Burgers equation is the following scalar conservation law
  considered in~\cite{JiangTadmo98}. In the framework leading
  to~\eqref{eq:1}, set
  $
  L=1,$ $N=2,$ $ 
  g(x,t,u) = u,$ 
  $f(x,t, u ) = \frac 1 2 u^2  \begin{bmatrix} 1 & 1\end{bmatrix},
  $ and 
  $b \equiv 0  $
  to get
  ${\partial}_t u + \frac 1 2 \left(  {\partial}_1 (u^2)  +  {\partial}_2 (u^2) \right) = 0$.
  Applying Theorem~\ref{thm:map} to map this equation from a tent
  $K_i$ to the spacetime cylinder $\hat K_i$, we find that 
  $
  \hat U = \hat G( \hat u) =  \hat u - 
  \frac 1 2 \hat u^2 
  ( \hat {\partial}_1 \varphi + \hat {\partial}_2 \varphi).
  $
  To illustrate how to use the second approach, we
  compute $ \hat G^{-1}(\hat U)$ by solving the
  quadratic equation 
  $
  d \hat u^2 - 2 \hat u + 2 \hat U =0
  $
  where $d = \hat {\partial}_1 \varphi + \hat {\partial}_2 \varphi.$  The roots are 
  $ \hat u = ( 1 \pm \sqrt{ 1 - 2 d \hat U} )/d$.
  In order to choose between the two roots, we now assume that the
  tents are constructed so that
  \begin{equation}
    \label{eq:12}
    | \hat u d |  < 1
  \end{equation}
  throughout $\hat K_i$.  Note that since $\hat u$ is the wave speed and
  $d$ is related to the tent pole height, this is the causality constraint.
  Note also that~\eqref{eq:12} implies that $\hat u d -1 \ne 0$,
  a necessary condition for the mapped system to be hyperbolic in the
  $\hat t$-direction -- cf.~\eqref{eq:hyperdet}. 

  Now, since~\eqref{eq:12} implies that
  $ \hat u d -1 \le | \hat u d | -1 < 0$, the only root that makes
  sense is the one satisfying
  $ \hat u d - 1 = - \sqrt{ 1 - 2 d \hat U} < 0$. Simplifying this
  root, we obtain
  \[
  \hat G^{-1}( \hat U ) = \frac{2 \hat U}{ 1 + \sqrt{ 1 - 2 d \hat U}}.
  \]
  One can now proceed with the second
  approach by applying a standard spatial discontinuous Galerkin
  discretization and time stepping by a Runge-Kutta scheme. Some
  regularization or slope limiting technique is needed to avoid
  spurious oscillations near sharp solution transitions. This issue is
  considered further in Section~\ref{sec:euler}.
\end{example}

\section{A locally implicit MTP scheme for the wave equation}  
\label{sec:wave}

\subsection{The acoustic wave problem}

Suppose we are given a material coefficient
$\alpha: \om_0 \to \RRR^{N \times N},$ symmetric and positive definite 
everywhere
in~$\om_0$ and a damping coefficient $\beta: \om_0 \to \RRR$. The
wave equation for the linearized pressure $\phi: \om \to \RRR$ is
\begin{subequations} \label{eq:waveIBVP}
\begin{equation}
  \label{eq:7}
  {\partial}_{tt} \phi + \beta {\partial}_t \phi - \divx (\alpha\gradx \phi) = 0 
  \qquad \text{in } \om.
\end{equation}
While a variety of  initial and boundary conditions 
are admissible in MTP schemes,
for definiteness, we focus on these model conditions:
\begin{align}
  \label{eq:16}
  n_x \cdot \alpha \gradx \phi & = 0  
  &&  \text{on } {\partial} \om_0 \times (0, \tmax),\\
  {\partial}_t\phi = \phi_1 \text{ and } 
  \phi & = \phi_0 &&  \text{on } \om_0 \times \{0\}.
\end{align}
\end{subequations} 
for some given sufficiently smooth compatible data $\phi_0$ and
$\phi_1$. In~\eqref{eq:16}, $n_x$ denotes the spatial component of the
outward unit normal.

Let us put~\eqref{eq:waveIBVP} into the
framework of~\eqref{eq:1} using Example~\ref{eg:linear}.  Set $L=N+1$ and
\[
u = 
\begin{bmatrix}
  q \\ \mu 
\end{bmatrix}
= 
\begin{bmatrix}
  \alpha \gradx \phi 
  \\ 
  {\partial}_t \phi
\end{bmatrix} 
\in \RRR^L.
\]
Then~\eqref{eq:7} yields 
$  \alpha^{-1} {\partial}_t q - \gradx \mu =0$ and 
$  {\partial}_t \mu - \dive q - \beta \mu = 0$.
This is readily identified to be in the form~\eqref{eq:5} with 
\[
\At =
\begin{bmatrix}
  \alpha^{-1} & 0 \\
  0 & 1 
\end{bmatrix}, \quad 
\Aj = -
\begin{bmatrix}
  0 & e_j \\
  e_j^t & 0 
\end{bmatrix}, \quad 
B = 
\begin{bmatrix}
  0 & 0 \\
  0 & \beta 
\end{bmatrix},
\]
where $e_j$ denotes the $j$th unit (column) vector.  The boundary
condition in the new variable is $ n_x \cdot q = 0$ on
${\partial}\om_0 \times (0,\tmax),$ and the initial conditions take the form
$ q = \alpha \gradx \phi_0$ and $ \mu = \phi_1$ on $\om_0.$

To describe the MTP scheme,  set
$u_0 = (q_0,\mu_0) = (\alpha \gradx \phi_0, \phi_1)$. 
Suppose we are at the $i$th tent pitching step. Then the 
solution $u_{i-1} = (q_{i-1}, \mu_{i-1})$ has been computed on the
advancing front $S_{i-1}$, and a new tent $K_i$ has been erected at
mesh vertex~$\vi$.  We now need the wave equation mapped over to
$\hat K_i = \om_{\vi} \times (0,1)$.  From 
Example~\ref{eg:linear}, 
\begin{equation}
  \label{eq:8}
\frac{{\partial}}{{\partial} \hat t} ( H \hat u )
+ 
\sum_{j=1}^N \frac{{\partial}}{{\partial} \hat x_j}( \delta \Aj \hat u ) 
+ \delta \hat B \hat u =0,
\end{equation}
where $H$ is as in~\eqref{eq:H} and $\hat B = B \circ \vPhi$ has the
sole nonzero entry $\hat \beta = \beta \circ \vPhi$.  In this example,
it is convenient to split $\hat u$ into two blocks consisting of
$\hat q = q \circ \vPhi \in \RRR^N$ and
$\hat\mu = \mu \circ \vPhi \in \RRR.$ Then for all
$(\hat x, \hat t ) \in \hat K_i$,
\begin{equation}
  \label{eq:H-wave}
H(\hat x, \hat t)
\begin{bmatrix}
  \hat q \\ \hat \mu 
\end{bmatrix}
= 
\begin{bmatrix}
  \hat\alpha^{-1}  & \gradx \varphi \\
  (\gradx \varphi)^t & 1
\end{bmatrix}
\begin{bmatrix}
  \hat q \\ \hat\mu 
\end{bmatrix}
\end{equation}
where $\hat \alpha = \alpha \circ \vPhi$ and~\eqref{eq:8} can be
rewritten as
\begin{equation}
  \label{eq:10}
\frac{{\partial}}{{\partial} \hat t} 
\begin{bmatrix}
  \hat{\alpha}^{-1} \hat q + \hat\mu \gradx \varphi 
  \\
  \hat\mu + \hat q \cdot \gradx \varphi 
\end{bmatrix}
- 
\begin{bmatrix}
  \gradx ( \delta \hat \mu)  
  \\
  \dive( \delta \hat q)
\end{bmatrix}
+ 
\begin{bmatrix}
  0 \\ 
  \delta \hat\beta \hat \mu
\end{bmatrix}
=0 \qquad
\text{ in } \om_{\vi} \times (0,1).
\end{equation}
On the cylinder, this equation must be supplemented by the initial
conditions $ \hat q = \hat{q}_{i-1}$ and $\hat \mu = \hat{\mu}_{i-1}$
on $\om_{\vi} \times \{0\}.$

\subsection{Semidiscretization after mapping}

For the spatial discretization, we use the Brezzi-Douglas-Marini (BDM)
mixed method. Namely, letting $P_p(T)$ denote the space of polynomials
of degree at most~$p$ in $\hat x$, restricted to a spatial
$N$-simplex~$T$, set
$X_i = \{ (r,\eta) \in \Hdiv{\om_{\vi}} \times L^2(\om_{\vi}) : r|_T
\in P_p(T)^N$
and $ \eta|_T \in P_p(T)$ for all simplices $T \in \TT_i$ and
$r \cdot n_x =0$ on ${\partial} \om_{\vi} \cap {\partial}\om_0\}.$
Multiplying~\eqref{eq:10} by $(r, \eta)$ and integrating the first
equation by parts, we obtain
\begin{align} \label{eq:19} 
  \frac{d}{d \hat t} 
  \int_{\om_\vi}
  \begin{bmatrix}
    \hat{\alpha}^{-1} \hat q + \hat\mu \gradx \varphi 
    \\
    \hat\mu + \hat q \cdot \gradx \varphi 
  \end{bmatrix}
  \cdot 
  \begin{bmatrix}
    r \\ \eta 
  \end{bmatrix}
  \, d \hat x 
  =
  \int_{\om_{\vi}}
  \begin{bmatrix}
  -\delta \hat \mu
  \\
  \dive( \delta \hat q)
  - \delta \hat\beta \hat \mu
  \end{bmatrix}
  \cdot 
  \begin{bmatrix}
  \dive r
  \\
  \eta 
  \end{bmatrix}
  \, d \hat x,
\end{align}
for all $(r, \eta) \in X_i$.  Using a basis
$\psi_m \equiv (r_m,\eta_m)$ of $X_i$, the 
coefficients $ \uu_m(\hat t)$ of
the expansion of $\hat u$ in this basis satisfy an ODE system, which
can be written using matrices $\mH$ and $\mS$ defined by
\begin{subequations}
  \label{eq:wave-semidiscrete}
\begin{align}
  \label{eq:wave-semidiscrete-H}
  \mH_{lm}(\hat t) 
  & = 
  \int_{\om_\vi}
  \begin{bmatrix}
    \hat{\alpha}^{-1} r_m + \eta_m \gradx \varphi 
    \\
    \eta_m +  r_m \cdot \gradx \varphi 
  \end{bmatrix}
  \cdot 
  \begin{bmatrix}
    r_l \\ \eta_l
  \end{bmatrix}
  \, d \hat x
  \\   
  \label{eq:wave-semidiscrete-S}
  \mS_{lm}
  & = 
  \int_{\om_{\vi}}
  \begin{bmatrix}
  -\delta \eta_m
  \\
  \dive( \delta r_m)
  - \delta \hat\beta \eta_m
  \end{bmatrix}
  \cdot 
  \begin{bmatrix}
  \dive r_l
  \\
  \eta_l
  \end{bmatrix}
  \, d \hat x.
\end{align}
Using prime $(')$ to abbreviate $d/ d\hat t$, observe that~\eqref{eq:19}
is the same as
\begin{equation}
\label{eq:wave-semidiscrete-ODE}
\left( \mH(\hat t)  \uu(\hat t) \right)'  = 
\mS \uu(\hat t), \qquad 0 < \hat t < 1,
\end{equation}
\end{subequations}
a realization of~\eqref{eq:H-approach-1} for the wave equation.

\subsection{Time discretization after mapping}

We utilize the first approach of \S\ref{ssec:first}. 
by applying an implicit high order multi-stage Runge-Kutta
(RK) method of Radau IIA type \cite[Chapter~IV.5]{HaireWanne91} for
time stepping~\eqref{eq:wave-semidiscrete-ODE}.  Note that due to the
implicit nature of the scheme, there is no CFL constraint on the
number of stages (within the mapped tent), irrespective of the spatial
polynomial degree~$p$ of $X_i$.
These RK methods, with $s$ stages, are characterized by numbers
$a_{lm}$ and $c_l$ for $l,m=1,\ldots, s$ (forming entries of a Butcher
tableau) with the property that $c_s =1$ (and the remaining $c_l$ are
determined by the roots of appropriate Jacobi polynomials).  When
applied to a standard ODE $y' = f(\hat t, y)$ in the interval
$\hat t \in (0,1)$, with initial condition $y(0)=y_0$, it produces
approximations $y_l$ to $y$ at $t_l=c_l$ that satisfy
\begin{equation}
  \label{eq:11}
y_l = y_0 +  \sum_{m=1}^s a_{lm} f(t_m,y_m), \qquad l =1,\ldots, s.  
\end{equation}
However, since~\eqref{eq:wave-semidiscrete-ODE} is not in this
standard form, we substitute $y_l = \mH_l \uu_l$ into~\eqref{eq:11},
where $\mH_l=\mH(\hat t_l)$ and $\uu_l$ is the approximation to
$\uu( t_l)$ to be found. Also setting $f(t_m,y_m) = \mS \uu_m$, we
obtain the linear system
\[
\mH_l \uu_l = \mH_0 \uu_0 + 
\sum_{m=1}^s a_{lm} \mS \uu_m, \qquad l =1,\ldots, s,
\]
which can be easily solved for the final stage solution~$\uu_s$, given
$\uu_0$.

\subsection{Numerical studies in two and three space dimensions}

The  locally implicit MTP method was implemented 
within the framework of 
 the NGSolve~\cite{Schob14} package. We report the results 
obtained for~\eqref{eq:waveIBVP} with $\beta=0,$
$\alpha=1$,  $\om_0$ set to the unit square, $\phi_0=0$ and
$\phi_1 = \cos(\pi x_1) \cos(\pi x_2)$ for $(x_1, x_2) \in \om_0$.  It
is easy to see that the exact solution is the classical standing wave
$\phi(x,t) = \cos(\pi x_1) \cos (\pi x_2) \sin( \pi t \sqrt{2})/
(\sqrt{2} \pi),$ 
\[
u(x,t) = 
\begin{bmatrix}
  q(x,t) \\ \mu(x,t)
\end{bmatrix} 
=
\begin{bmatrix}
  \gradx \phi \\
  {\partial}_t \phi
\end{bmatrix}
= 
\begin{bmatrix}
  -\sin (\pi x_1) \cos( \pi x_2) \sin( \pi t \sqrt{2})/ \sqrt{2} \\
  -\cos (\pi x_1) \sin( \pi x_2) \sin( \pi t \sqrt{2})/ \sqrt{2} \\
  \cos(\pi x_1) \cos(\pi x_2) \cos( \pi t \sqrt{2}) 
\end{bmatrix}.
\]
The spatial domain $\om_0$ is meshed by a uniform grid obtained by
dividing the unit square into $2^l \times 2^l$ congruent squares and
dividing each square into two triangles by connecting its positively
sloped diagonal. The parameters to be varied in each experiment are
the spatial mesh size~$h=2^{-l}$ and the the polynomial degree $p$ of
the space discretization. The number of Runge-Kutta time stages is
fixed to $s=p$.  The tent meshing algorithm is driven by an input
wavespeed of 2 (leading to conservative tent pole heights) to mesh a
time slab of size $2^{-l}/8$. This time slab is stacked in time to
mesh the entire spacetime region of simulation $\om_0 \times (0,1)$.
Letting $q_h(x)$ and $\mu_h(x)$ denote the computed solutions at time
$t=1$, we measure the error norm $e$ defined by
$ e^2 = \| q(\cdot, 1) - q_h \|^2_{L^2(\om_0)} + \| \mu(\cdot,1) -
\mu_h \|_{L^2(\om_0)}^2.$ The observations are compiled in
Figure~\ref{fig:errp}, where the values of $e$ as a function of degree
$p$ and $h$ are plotted.  The rate $r$ of the $O(h^r)$-convergence
observed is computed from the slope of the regression lines and marked
near each convergence curve.  We observe that $e$ appears to go to $0$
at a rate of $O(h^p).$

\begin{figure}  
  \centering
  \pgfplotsset{yticklabel style={text width=2.5em,align=right}}
  \begin{subfigure}[t]{0.45\textwidth}
  \begin{tikzpicture}
    \begin{loglogaxis}[
      width=1.05\textwidth,
      height=0.37\textheight,
      xlabel=$h$,
      ylabel=$e$,
      legend pos = south east,
      legend entries = {$p=1$, $p=2$,$p=3$, $p=4$, $p=5$ },
      ]

      \addplot table[x=h,y=e_p1s1] {tab2d.dat};
      \addplot table[x=h,y=e_p2s2] {tab2d.dat};
      \addplot table[x=h,y=e_p3s3] {tab2d.dat};
      \addplot table[x=h,y=e_p4s4] {tab2d.dat};
      \addplot table[x=h,y=e_p5s5] {tab2d.dat};

      \addplot[dotted] 
      table[x=h,y={create col/linear regression={y=e_p1s1}}]            
            {tab2d.dat}
            coordinate [pos=0.97] (A)
            coordinate [pos=0.81] (B);
            \xdef\slopeSqrP{\pgfplotstableregressiona}
            \draw[dotted] (A) -| (B) node [pos=0.75,anchor=west]
                         {\pgfmathprintnumber[fixed, precision=1]
                           {\slopeSqrP}};
      \addplot[dotted] 
      table[x=h,y={create col/linear regression={y=e_p2s2}}]            
            {tab2d.dat}
            coordinate [pos=0.97] (A)
            coordinate [pos=0.81] (B);
            \xdef\slopeSqrPP{\pgfplotstableregressiona}
            \draw[dotted] (A) -| (B) node [pos=0.75,anchor=west]
                         {\pgfmathprintnumber[fixed, precision=1]
                           {\slopeSqrPP}};
      \addplot[dotted] 
      table[x=h,y={create col/linear regression={y=e_p3s3}}]            
            {tab2d.dat}
            coordinate [pos=0.97] (A)
            coordinate [pos=0.81] (B);
            \xdef\slopeSqrPPP{\pgfplotstableregressiona}
            \draw[dotted] (A) -| (B) node [pos=0.75,anchor=west]
                         {\pgfmathprintnumber[fixed, precision=1]
                           {\slopeSqrPPP}};
      \addplot[dotted] 
      table[x=h,y={create col/linear regression={y=e_p4s4}}]            
            {tab2d.dat}
            coordinate [pos=0.97] (A)
            coordinate [pos=0.81] (B);
            \xdef\slopeSqrPPPP{\pgfplotstableregressiona}
            \draw[dotted] (A) -| (B) node [pos=0.75,anchor=west]
                         {\pgfmathprintnumber[fixed, precision=1]
                           {\slopeSqrPPPP}};
      \addplot[dotted] 
      table[x=h,y={create col/linear regression={y=e_p5s5}}]            
            {tab2d.dat}
            coordinate [pos=0.97] (A)
            coordinate [pos=0.81] (B);
            \xdef\slopeSqrPPPPP{\pgfplotstableregressiona}
            \draw[dotted] (A) -| (B) node [pos=0.75,anchor=west]
                         {\pgfmathprintnumber[fixed, precision=1]
                           {\slopeSqrPPPPP}};
    \end{loglogaxis}
  \end{tikzpicture}  
  \caption{Example in two space dimensions}
  \label{fig:errp}
  \end{subfigure}
  \qquad\quad
  \begin{subfigure}[t]{0.45\textwidth}
  \begin{tikzpicture}  
    \begin{loglogaxis}[
      width=1.05\textwidth,
      height=0.37\textheight,
      xlabel=$h$,
      ylabel=$e$,
      legend pos = south east,
      legend entries = {$p=1$, $p=2$, $p=3$} 
      ]

      \addplot table[x=h,y=e_p1s1] {tablerr3d.dat};
      \addplot table[x=h,y=e_p2s2] {tablerr3d.dat};
      \addplot table[x=h,y=e_p3s3] {tablerr3d.dat};

      \addplot[dotted] 
      table[x=h,y={create col/linear regression={y=e_p1s1}}]            
            {tablerr3d.dat}
            coordinate [pos=0.97] (A)
            coordinate [pos=0.81] (B);
            \xdef\slopeSqrP{\pgfplotstableregressiona}
            \draw[dotted] (A) -| (B) node [pos=0.75,anchor=west]
                         {\pgfmathprintnumber[fixed, precision=1]
                           {\slopeSqrP}};
      \addplot[dotted] 
      table[x=h,y={create col/linear regression={y=e_p2s2}}]            
            {tablerr3d.dat}
            coordinate [pos=0.97] (A)
            coordinate [pos=0.81] (B);
            \xdef\slopeSqrPP{\pgfplotstableregressiona}
            \draw[dotted] (A) -| (B) node [pos=0.75,anchor=west]
                         {\pgfmathprintnumber[fixed, precision=1]
                           {\slopeSqrPP}};
      \addplot[dotted] 
      table[x=h,y={create col/linear regression={y=e_p3s3}}]            
            {tablerr3d.dat}
            coordinate [pos=0.97] (A)
            coordinate [pos=0.81] (B);
            \xdef\slopeSqrPPP{\pgfplotstableregressiona}
            \draw[dotted] (A) -| (B) node [pos=0.75,anchor=west]
                         {\pgfmathprintnumber[fixed, precision=1]
                           {\slopeSqrPPP}};
    \end{loglogaxis}
  \end{tikzpicture}  
  \caption{Example in three space dimensions}
  \label{fig:errp3d}
  \end{subfigure}
  \caption{Convergence rates for a standing wave}
  \label{fig:err}
\end{figure}
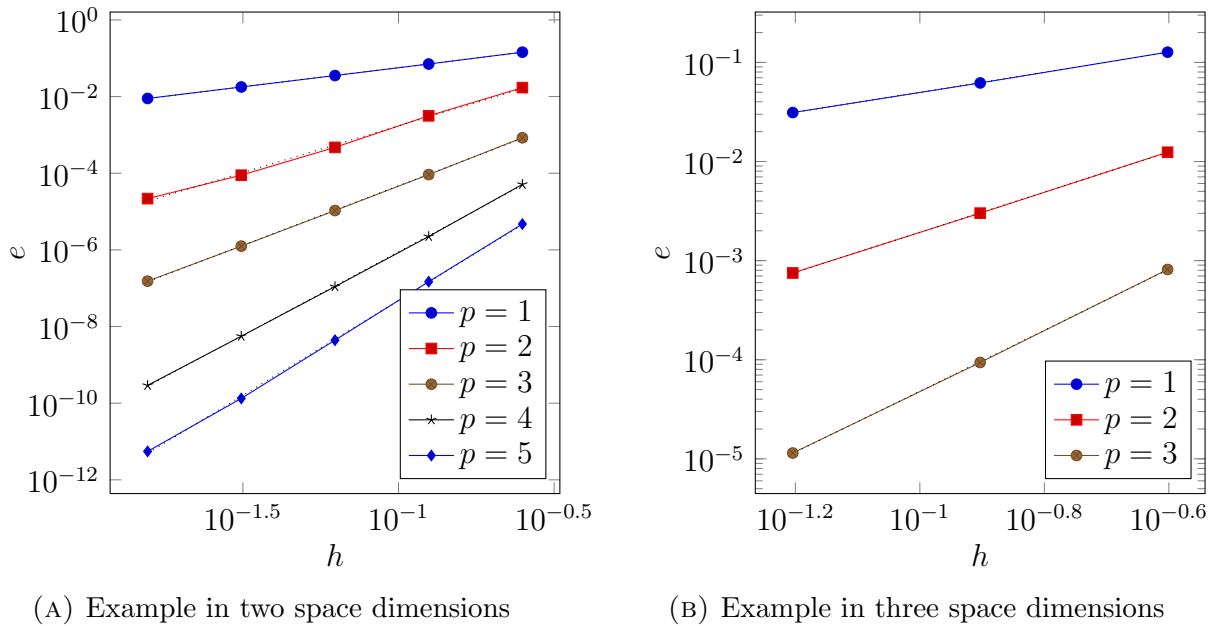

Next, consider the case of three spatial dimensions, where $\om_0$ is
set to the unit cube and subdivided in a fashion analogous to the
two-dimensional case (into $2^l \times 2^l \times 2^l$ congruent
cubes, which are further subdivided into six tetrahedra).  The
remaining parameters are the same as in the two-dimensional case,
except that now the exact solution is
$\phi(x,t) = \cos(\pi x_1) \cos (\pi x_2) \cos (\pi x_3) \sin( \pi t
\sqrt{3})/ (\sqrt{3} \pi)$.  Note that the spacetime mesh of
tents, now formed by {\em four-dimensional} simplices, continues to be
made by Algorithm~\ref{alg:pitchloc}.  
The convergence history plotted in Figure~\ref{fig:errp3d} shows that
$e$, just as in the previous case, goes to zero at a rate of
$O(h^p).$

\section{An explicit MTP scheme for a nonlinear conservation law} 
\label{sec:euler}

In this section, we describe some techniques for handling nonlinear
conservation laws, and considering the specific example of Euler
equations, construct a explicit MTP scheme.

\subsection{Mapping an entropy pair}

Recall that a real function ${\EE}(u)$ is called an entropy
\cite[Definition~3.4.1]{Serre99} of the system~\eqref{eq:1} if there
exists an entropy flux $\FF(u) \in \RRR^N$ such that every classical
solution $u$ of~\eqref{eq:1} satisfies ${\partial}_t {\EE}(u) + \divx \FF(u) = 0.$
Note that for nonsmooth~$u$, this equality need not hold.  The pair
$({\EE},{\FF})$ is called the entropy pair. We say that this pair
satisfies the ``entropy admissibility condition'' on $\om$ if 
\begin{equation}
  \label{eq:18}
{\partial}_t {\EE}(u(x,t)) + \divx {\FF}(u(x,t) ) \le 0   
\end{equation}
holds in the sense of distributions on $\om$.  The inequality is
useful to study the violation of entropy conservation for nonsmooth
solutions (like shocks). Nonlinear conservation laws often have
multiple weak solutions and uniqueness is obtained by selecting a
solution $u$ satisfying the entropy admissibility condition. These
theoretical considerations motivate the use of numerical analogues
of~\eqref{eq:18} in designing schemes for conservation laws.

Suppose that on a tent $K_i$, we are given a solution $u(x,t)$
of~\eqref{eq:1} and an entropy pair $({\EE},{\FF})$.  The mapped solution, as
before, is $\hat u = u \circ \vPhi$. Define
\begin{align} \label{eq:EFhat}
  \hat {\EE} (w)
  & = {\EE} (w) - {\FF}(w) \gradx \varphi,
  &
  \hat {\FF} ( w )
  & = 
    \delta {\FF}(w).
\end{align}

\begin{theorem} \label{thm:mappedentropy}
  Suppose $u$ solves~\eqref{eq:1} on $K_i$ and
  $\hat u = u \circ \vPhi$ solves the mapped
  equation~\eqref{eq:mapped}.  Then, whenever $({\EE},{\FF})$ is an entropy
  pair for~\eqref{eq:1}, $(\hat {\EE}, \hat {\FF})$ is an entropy pair
  for~\eqref{eq:mapped}. Moreover, if ${\EE}(u)$ and ${\FF}(u)$ satisfies the
  entropy admissibility condition~\eqref{eq:18} on $K_i$, then
  $\hat {\EE}(\hat u)$ and $\hat {\FF}(\hat u)$ satisfies the entropy
  admissibility condition 
  on $\hat K_i$.
\end{theorem}
\begin{proof}
  Repeating the calculations
  in the proof of Theorem~\ref{thm:map}, with $g = {\EE}$ and $f = {\FF}$, 
  we  obtain 
  \[
  ({\partial}_t {\EE} (u) + \divx {\FF}(u)) \circ \vPhi
  = 
  \frac{1}{\delta}
  \left(
    \hat {\partial}_t \hat {\EE} (\hat u) + \divxh \hat {\FF} (\hat u)
  \right),
  \]
  from which the statements of the theorem follow.
\end{proof}

\subsection{Entropy viscosity regularization} \label{ssec:entrvisc}

The addition of ``artificial viscosity'' (a diffusion term) to the
right hand side of nonlinear conservation laws makes their solutions
dissipative.  When the limit of such solutions, as the diffusion term
goes to zero, exist in some sense, it is referred to as a vanishing
viscosity solution. It is known \cite[Theorem~4.6.1]{Dafer10} that the
vanishing viscosity solution satisfies the entropy admissibility
condition for entropy pairs satisfying certain conditions.  
Motivated by such connections, 
the
entropy viscosity regularization method of~\cite{GuermPasquPopov11},
suggests modifying numerical schemes by
selectively adding small amounts of artificial viscosity, to avoid
spurious oscillations near discontinuous solutions. We borrow this
technique and incorporate it  into the MTP schemes obtained using the
second approach (of \S\ref{ssec:second-approach}) as follows.

Consider the problem on the tent $K_i$ mapped to $\hat K_i$. We set
the spatial discretization space to
$X_i = \{ u \in L^2(\om_\vi)^L: u|_T \in P_p(T) $ for all
$T \in \TT_i \}$ and consider a DG discretization of the mapped
equation~\eqref{eq:mapped} following the second approach. Accordingly
the approximation $\hat U_h(\hat x, \hat t)$ takes the form
in~\eqref{eq:secondapp}. Let $(\cdot, \cdot)_h$ and
$\ip{ \cdot, \cdot}_h$ denote the sum of integrals over $T$ and ${\partial} T$
of the appropriate inner product of its arguments, over all
$T \in \TT_i$, respectively.
The semidiscretization of~\eqref{eq:mapped}
by the DG method takes the form
\begin{equation}
  \label{eq:DGsemi}
  (\hat {\partial}_t \hat U_h, V)_h 
  -( \delta \, f( \hat G^{-1}(\hat U_h)), \gradx V )_h
  + \ip{ \delta Q_f(\hat U_h), V }_h
  + (\delta \, b , V)_h = 0 
\end{equation}
for all $V \in X_i$. Here $Q_f$ is the so-called ``numerical flux,''
whose form varies depending on the DG method, and as usual, all
derivatives are taken element by element.

Suppose that an entropy pair $({\EE},{\FF})$ is given for~\eqref{eq:1}. On the
mapped tent $\hat K_i$, let $(\hat {\EE}, \hat {\FF})$ be defined
by~\eqref{eq:EFhat}.  Suppose a numerical approximation
$\hat U_h(\hat x, \hat t_1)$ has been computed at some time
$0\le \hat t_1 <1$ and we want to compute a numerical approximation at
the next time stage, say at $\hat t = \hat t_1 + \Delta t \le 1$.  The
{\em entropy residual} of the approximation
$u_h = \hat G^{-1}(\hat U_h)$ to $u$ is a weak form of the quantity
$\hat{\partial}_t\hat {\EE}( \hat u_h) + \divxh \hat {\FF}( \hat u_h)$, which by
Theorem~\ref{thm:mappedentropy}, is non-positive. The discrete entropy
residual at time $\hat t_1$ is $R_h = \min( r_h,0)$ where
$r_h \in X_i$ is defined by
\begin{align*}
(\delta r_h, V)_h 
  & = (\hat {\partial}_t \hat {\EE}( \hat G^{-1} (\hat U_h) ),  V)_h
- (\hat {\FF}( \hat U_h), \gradx V)_h + 
\ip{ \delta Q_{\FF} (\hat U_h), V}_h
  \\
  & = 
(\frac{{\partial} (\hat {\EE} \circ \hat G^{-1})}{{\partial} U} \hat {\partial}_t\hat U_h,  V)_h
- (\hat {\FF}( \hat U_h), \gradx V)_h + 
\ip{ \delta Q_{\FF} (\hat U_h), V}_h
\end{align*}
for all $V \in X_i$.  Here  $Q_{\FF}$ is a numerical flux
prescribed by a DG approximation to the entropy conservation
equation. The term $\hat {\partial}_t \hat U_h$ can be replaced by its
approximation available from~\eqref{eq:DGsemi} while computing $r_h$.

Next, following~\cite{GuermPasquPopov11}, we quantify the amount of
viscosity to be added to~\eqref{eq:DGsemi}. Define the {\em entropy
  viscosity coefficient} on one spatial element $T\in \TT_i$ by
$
  \nu_e^T =
  c_X^2  \| R_h \|_{L^\infty(T)}/ |\bar {\EE}|
$
where $\bar {\EE}$ is the mean value of $\hat {\EE}( \hat G^{-1} (\hat U_h))$
on $T$ and $c_X$ is an effective local grid size of~$X_i$, typically
chosen as
$
c_X = \kappa_1
 \diam(T)/p
$
for some fixed number $\kappa_1$.  To limit the viscosity added based
on local wavespeed, define
$  \nu_*^T = \kappa_2  \diam(T) 
  \| D_u \hat f( \hat x, \hat t_1, \hat u_h(\hat x, \hat t_1) ) \|_{L^\infty(T)}
$
where $\kappa_2$ is another fixed number and set 
$
\nu_i = \max_{T \in \TT_i}  \min (\nu_*^T, \nu_e^T).
$
This artificial viscosity coefficient proposed
in~\cite{GuermPasquPopov11} leads to generous viscosity at
discontinuities (where the entropy residual is high) and little
viscosity in smooth regions. Finally, we modify the mapped
equation~\eqref{eq:mapped} by adding to its right hand side the
corresponding artificial viscosity term
$\nu_i \divxh ( \delta \gradx \hat u)$.  Namely, instead
of solving~\eqref{eq:DGsemi} for
$\hat t_1 \le \hat t \le \hat t_1 + \Delta t,$ we solve its viscous
perturbation:
\begin{equation}
  \label{eq:mappedev}
  \begin{aligned}
  (\hat {\partial}_t \hat U_h, V)_h 
  -( \delta \, f( \hat G^{-1}(\hat U_h)), \gradx V )_h
  & + \ip{ \delta Q_f(\hat U_h), V }_h
  + (\delta \, b , V)_h
  \\
  & + 
  \nu_i a_i( \hat G^{-1}(\hat U_h), V) = 0,
  \end{aligned}
\end{equation}
for all $v \in X_i$, where $a_i(\cdot,\cdot)$ is the standard interior
penalty DG approximation of the viscous term
$-\divxh (\delta \gradx \hat u)$ defined below.  On an interface $F$
shared by two elements $T_+$ and $T_-$, with outward unit normals
$n_+$ and $n_-$, respectively, set $[wn]=w|_{T_+} n_ + + w|_{T-} n_-$,
with the understanding that $w(x,t)$ is considered zero if $x$ is
outside $\om_{\vi}$. Then
\[
a_i(w,v) =  (\delta  \gradx w, \gradx v)_h
-\frac 1 2 \ip{\delta \gradx w, [vn]}_h 
-\frac 1 2 \ip{[wn],\delta \gradx v}_h
+ \frac{\alpha}{2 h } \ip{ \delta[wn],[vn]}_h.
\]
Here, as usual, the penalization parameter $\alpha$ must be chosen
large enough to obtain coercivity.  Applying a time stepping algorithm
to~\eqref{eq:mappedev}, we compute the numerical solution at the next
time stage $ t_1 + \Delta t$.

\subsection{Application to Euler equations}

Let $\rho: \om \to \RRR$, $m:\om \to \RRR^N$ and $E: \om \to \RRR$
denote the density, momentum, and total energy of a perfect gas
occupying $\om \subset \RRR^N$.  Set $L = N+2$ and let
\[
u=
\begin{bmatrix}
  \rho \\ m \\ E
\end{bmatrix}, 
\quad g( u) = u\,,
\quad
f( u) = 
\begin{bmatrix}
  m
  \\
  \pp I + m \otimes m/ \rho 
  \\
  ( E + \pp)  m / \rho
\end{bmatrix}, 
\quad b \equiv 0\,,
\]
Here, the pressure $\pp$ is related to the state variables by 
$
\pp = \frac{1}{2}\rho \tT,
$ and $
\tT = \frac 4 d ( 
\frac{E}{\rho} - \frac 1 2 \frac{|m|^2}{\rho^2}),
$
where $d$, the degrees of freedom of the gas particles, is set to $5$
for ideal gas.  With these settings, the system of Euler equations is
given by~\eqref{eq:1}.

After mapping from a tent $K_i$ to $\hat K_i$, to proceed with the
second approach we need to invert the nonlinear equation
$\hat U = \hat G (\hat u)$. Namely, writing
$\hat u = (\hat \rho, \hat m, \hat E)$ and
$\hat U = ( \hat R, \hat M, \hat F)$, we want to explicitly compute
$ (\hat \rho, \hat m, \hat E) = \hat G^{-1} ( \hat R, \hat M, \hat F)$.
Lengthy calculations (see~\cite{Winte15}) show that the expression for
$\hat G^{-1}$ is given by
\begin{align*}
  \hat \rho 
  & = \ds{ \frac{\hat R^2}{ a_1 - \frac{2}{d}|\gradx \varphi|^2 a_3}},
   &
  \hat m
  & = \ds{\frac{\hat \rho }{\hat R} ( \hat M +  \frac 2 d a_3 \gradx \varphi  )},
   &
  \hat E 
  & = 
   \ds{\frac{\hat \rho }{\hat R}
   ( \hat F + \frac{2a_3}{d\hat \rho} \gradx \varphi \cdot \hat m  )}
\end{align*}
where 
$
a_1 = \hat R - \hat M \cdot \gradx \varphi,
$
$
a_2= 2 \hat F \hat R - |\hat M|^2,
$ and $
a_3= a_2/( a_1 + \sqrt{a_1^2 - \frac{4(d+1)}{d^2} |\gradx \varphi|^2 a_2)}.
$
The well-known expressions for the entropy and entropy flux for the
Euler system are
$
\EE( \rho,m,E) = \rho \left(
   \ln \rho - \frac{d}{2} \ln \tT
\right)$ and $
\FF(\rho,m,E) = m \EE /\rho.
$
With these expressions we discretize the mapped equation using the
second approach, applying the previously described entropy viscosity
regularization of~\eqref{eq:mappedev}.

\subsection{A computational illustration}

 \begin{figure}
   \centering
   \begin{subfigure}[t]{0.9\textwidth}
     \begin{center}
     \begin{tikzpicture}[scale=3]
      \draw (0,0) -- (0.6,0) node[midway,above] {\tiny{reflect}}    
      -- (0.6,0.2)   node[near end,above,sloped] {\tiny{reflect}};

      \draw (0.6,0.2) 
      -- (3,0.2)     node[midway,above] {\tiny{reflect}} ;

      \draw (3,0.2)
      -- (3,1)       node[midway,sloped,above] {\tiny{outflow}}      
      -- (0,1)       node[midway,below] {\tiny{reflect}}
      -- (0,0)       node[sloped,midway,above] {\tiny{inflow}}; 

      \draw[mark=|,mark size=0.4pt,-latex] 
      plot coordinates {(0,0)} node [below]{\scriptsize $0$} -- 
      plot coordinates {(0.6,0)} node [below]{\scriptsize $0.6$} -- 
      plot coordinates {(3,0)} node[below]{\scriptsize $3$} -- (3.15,0) node[right] {$x_1$};
      \draw[mark=-,mark size=0.4pt,-latex] 
      plot coordinates {(0,0)} node [left]{\scriptsize $0$} -- 
      plot coordinates {(0,0.2)} node [left]{\scriptsize $0.2$} -- 
      plot coordinates {(0,1)} node [left]{\scriptsize $1$} -- (0,1.15) node[above] {$x_2$};
    \end{tikzpicture}
    \end{center}
    \caption{Geometry and  boundary conditions}
    \label{fig:mach3:geom}
  \end{subfigure}
  \begin{subfigure}{0.9\textwidth}
      \begin{tikzpicture}
        \node (C) at (1.3,3) {\includegraphics[height=1.2in]
          {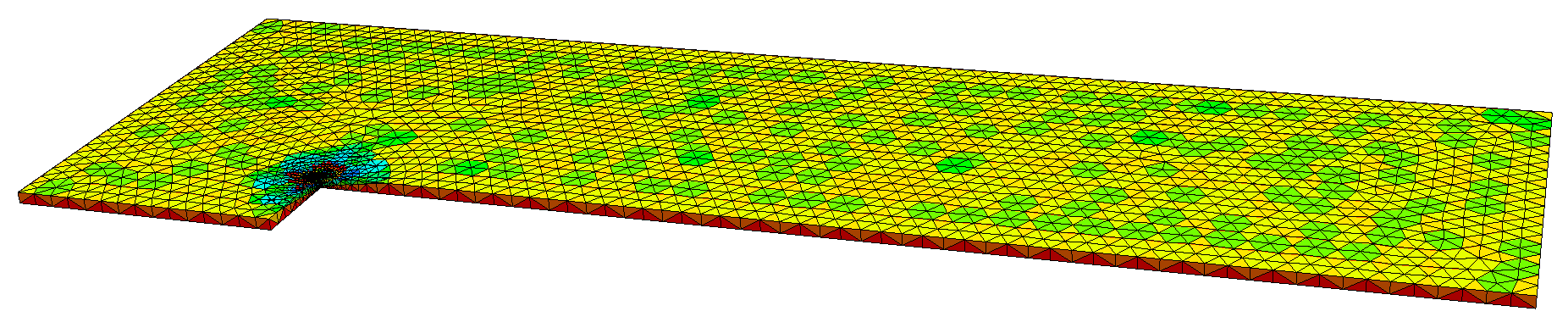}};
        {
          \node[circle,draw, inner sep=45,
          path picture={
            \node at ($(path picture bounding box.center)+(0,0.55)$) {
              \includegraphics[width=5.6cm]
              {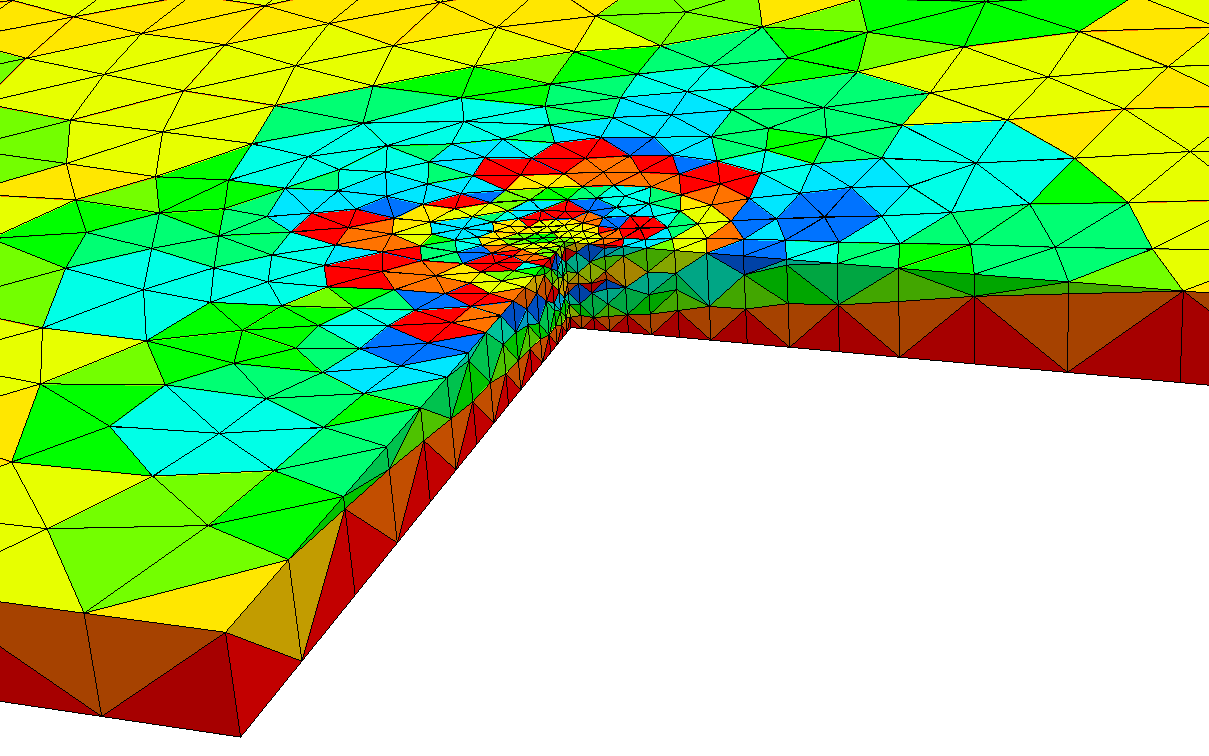}
            };
          }] (A) {};
          \path (A) (138.8:4.08) node[circle,draw,inner sep=10] (B) {};
          \draw[thick,<-] (A)--(B);
        }
      \end{tikzpicture}
      \caption{Locally refined tent mesh of one time slab and zoomed
        in view. (The tent colors cycle through the parallel layer numbers -- see Remark~\ref{rem:partents}.)}      \label{fig:mach3:mesh}
    \end{subfigure}
 \end{figure}
 
 We consider the well-known example~\cite{WoodwColel84} of a
 wind tunnel with a forward facing step on which a Mach 3 flow
 impinges. The geometry is shown in Figure~\ref{fig:mach3:geom} and
 the initial conditions are set to $ \rho = 1.4,$
 $m = \rho\begin{bmatrix} 3 & 0 \end{bmatrix}^t$, and $ \pp = 1.$ The
 boundary conditions are set such that $(0,x_2)$ is an inflow boundary
 and $(3,x_2)$ is a free boundary, which has no effect on the
 flow. All other boundaries are solid walls. 
 Anticipating the singularity at the nonconvex corner, we construct a
 spatial mesh with small elements near it.
 Figure~\ref{fig:mach3:mesh} shows this mesh and the unstructured
 locally adaptive time advance that is possible.

Using the notation of (\ref{eq:appr-2}) and a basis $\psi_l$ of $X_i$, we obtain the ODE system
$(\mM  \mU(\hat t))' = 
  \mR^1 (\mU(\hat t)) - \mR^2 (\mU(\hat t)),$
for $0 < \hat t < 1$,
where
$  [\mR^1]_l = 
  ( \delta  f( \hat G^{-1}(\hat U_h)), \gradx \psi_l )_h
  - \ip{
    \delta Q_f(\hat U_h), \psi_l }_h$ and 
$  [\mR^2]_l = 
  \nu_i a_i( \hat G^{-1}(\hat U_h), \psi_l)$.
This system within each tent is solved by a time stepping scheme and a
time step $\Delta t = \tfrac{1}{m}$, where $m$ denotes the number of local time
steps. For stability we need $m\ge O(p^2)$, but more time steps may be
used for accuracy.  Due to the addition of artificial viscosity, an
additional fractional time step $\Delta t_v$ is chosen depending on
the viscosity coefficient (and therefore on the smoothness of the
solution). A detailed algorithm based on the explicit Euler method can be found below, where we use
the notations $\mU^j := \mU(j\Delta t)$ and $\delta_* = \| \delta \|_{L^\infty(\om_{\vi})}$.

\begin{algorithm} 
For $j = 0,\dots,m-1$ do: 
\begin{itemize}
\item Evaluate $\mR^1(\mU^j)$.
\item Update solution $\mU^{j+1} = \mU^{j} + \Delta t \,\mR^1(\mU^j)$.
\item Calculate the entropy residual and the viscosity 
coefficient $\nu_i(j\Delta t)$.
\item Estimate time step $\Delta t_v = \Delta t / \frac{\delta_*\nu_i p^4}{h^2}$ for the artificial viscosity.
\item Apply the artificial viscosity with an explicit Euler method up to the time $(j+1)\Delta t$.
\end{itemize}
\end{algorithm}

This algorithm can be generalized for any Runge-Kutta scheme and for the following results a two-staged RK scheme was used.
A kinetic flux (see \cite{MandalDesh}) was used 
for the numerical flux $Q_f$ while   $Q_{\FF}$ was set by
\begin{align*}
  Q_{\FF} & =
  \begin{cases}
    \FF(\hat\rho^+,\hat m^+,\hat E^+)\cdot n\,, & \hat m^+\cdot n \ge 0\,, \\
    \FF(\hat\rho^-,\hat m^-,\hat E^-)\cdot n\,, & \text{otherwise}\,, \\
  \end{cases}
\end{align*}
where $\hat \rho^+$ denotes the trace of $\hat \rho$ from within the element which has $n$ as outward unit normal vector.
For computational convenience, we use a slight variation
of the 
entropy viscosity regularization described in \S\ref{ssec:entrvisc}. Namely,
the entropy viscosity coefficient on one element $T\in \TT_i$ is
set  by
$
  \nu_e^T = c_X^2 \| R_h \|_{L^\infty(T)}
$
and the limiting artificial viscosity is set by 
$
  \nu_*^T = \kappa_2  \diam(T) 
  \| \rho( |\frac{m}{\rho}| + \sqrt{\gamma \tT}) \|_{L^\infty(T)}
$
with $\gamma = \frac{d+2}{d} = 1.4$ for an ideal gas and the temperature $\tT$.
The constants in the calculation of the entropy viscosity coefficient were chosen as $\kappa_1 = \frac{1}{2}$, $\kappa_2 = \frac{1}{4p}$ and the penalization parameter $\alpha$ in the artificial viscosity term is set to 2.

\begin{figure}
  \centering
  \includegraphics[width=0.9\textwidth,trim=0 1450 380 0,clip]
  {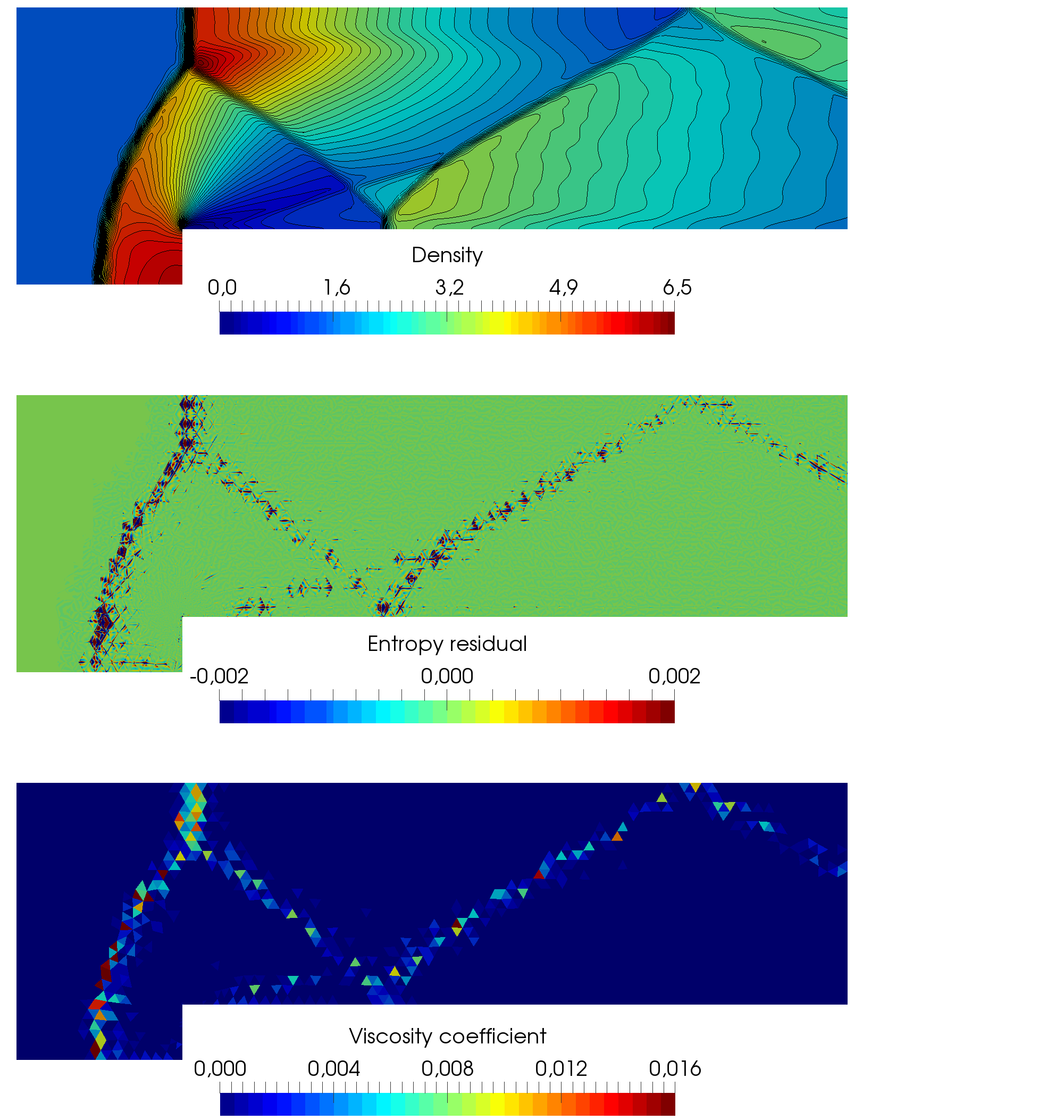}
  \includegraphics[width=0.9\textwidth,trim=0 0 380 1450,clip]
  {wind_tunnel_p4_t4_newvisc_kinflux2_iso50.png}
  \caption{Solution of Mach 3 wind tunnel at $t=4$, with $p=4$ DG
    finite elements on 3951 triangles}
  \label{fig:windtunnel_p4t4}
\end{figure}

With these settings, the results obtained with $p=4$, are shown in
Figure \ref{fig:windtunnel_p4t4}. They correspond to the
results~\cite{WoodwColel84} that can be found in the literature using
other methods. Note from the second plot that the artificial viscosity
is applied only in the shocks.

\section{Conclusion}

We have introduced new schemes, called MTP schemes, for advancing
hyperbolic solutions through unstructured tent meshes.  The advantages
of tent pitching over traditional time stepping, amply
clarified by others in the literature on SDG methods, include the
ability to advance in time by different amounts at different spatial
locations, easy parallelization, and linear scaling of computational
complexity in the number of tents. 

Further new advantages brought about by MTP schemes include the
possibility to use existing spatial discretizations and time stepping
schemes after mapping tents to cylinders. The mapping technique has
opened a new avenue to perform fully explicit matrix-free local
time stepping on unstructured tent meshes using explicit MTP schemes.
Their utility as a powerful computational tool was demonstrated on the
Mach~3 wind tunnel where local refinement near a rarefaction
singularity permitted us to capture the shock structure using 
relatively few elements and standard discretizations (in separated
space and time).

We also studied locally implicit MTP schemes and their application to
the acoustic wave equation. We observed $O(h^p)$ accuracy (in the
$L^2$ norm) when spatial basis functions of degree $p$ were used.  In
contrast, SDG schemes report $O(h^{p+1})$ accuracy if spacetime basis
functions of degree $p$ are used.  Thus in order to get the same
accuracy, MTP schemes use one higher spatial order. Despite this, the
locally implicit MTP scheme, due to its separation of time and space,
is more efficient than SDG schemes.  As $p$ increases, SDG schemes use
$O(p^{N+1})$ spacetime basis functions per tent, while MTP schemes use
$O((p+1)^N)$ spatial basis functions to obtain the same convergence
rate.  Hence to propagate the solution inside a tent, an SDG scheme
performs $O(p^{2(N+1)})$ flops, while the implicit MTP scheme performs
$O( (p+1)^{2N} )$ flops.  Since $(p+1)^{2N} < p^{2(N+1)}$ for $p\ge 3$
in both two and three space dimensions ($N=2,3$), the flop count
favors the implicit MTP scheme as $p$ increases.

Ongoing studies aim to provide rigorous proofs of the convergence
rates for the  MTP schemes~\cite{GSW16} and to provide
computational benchmarks for specific applications~\cite{mtpMax}. The
promises of improved performance of explicit MTP schemes, due to
better ratio of flops per memory (data locality) and matrix-free implementation
techniques (such as sum factorization algorithms), will be realized in
future works.  Years of research on SDG schemes have resulted in
advanced techniques like spacetime adaptive tent mesh refinement and
element-wise conservation. Further studies are needed to bring such
techniques to MTP schemes.


\end{document}